\documentclass{amsart}

\usepackage{amsfonts, amssymb, amsmath, amsthm}                               
\usepackage{slashed}                                                          

\newtheorem{theorem}{Theorem} 	      	      	                              
\newtheorem{corollary}[theorem]{Corollary}     	      	      	      	      
\newtheorem{lemma}[theorem]{Lemma}     	       	      	      	      	      
\newtheorem{proposition}[theorem]{Proposition} 	      	      	      	      
\newtheorem*{remark}{Remark}                                                  

\newcommand{\ol}[1]{\overline{#1}}                                            
\newcommand{\ul}[1]{\underline{#1}}                                           
\newcommand{\mf}[1]{\mathfrak{#1}}                                            
\newcommand{\mc}[1]{\mathcal{#1}}                                             

\newcommand{\paren}[1]{\left(#1\right)}                                       
\newcommand{\brak}[1]{\left[#1\right]}                                        
\newcommand{\brac}[1]{\left\{#1\right\}}                                      
\newcommand{\abs}[1]{\left|#1\right|}                                         
\newcommand{\brcmp}[1]{\left\langle#1\right\rangle}                           
\newcommand{\nabs}[1]{\left\|#1\right\|}                                      
\newcommand{\valat}[1]{\left.#1\right|}                                       

\newcommand{\R}{\mathbb{R}}                                                   
\newcommand{\Sph}{\mathbb{S}}                                                 

\DeclareMathOperator{\trace}{tr}                                              

\newcommand{\lapl}{\Delta}                                                    
\newcommand{\nasla}{\slashed{\nabla}} 	      	      	      	      	      
\newcommand{\lasl}{\slashed{\lapl}} 	      	      	      	      	      

\newcommand{\sss}[2]{#1^{\scriptscriptstyle (#2)}{}}                          

\allowdisplaybreaks

\begin{document}

\title[A Generalized Representation Formula]{A Generalized Representation Formula for Systems of Tensor Wave Equations}
\author{Arick Shao}
\address{Department of Mathematics, Princeton University, Princeton NJ 08544}
\email{aricks@math.princeton.edu}

\begin{abstract}
In this paper, we generalize the Kirchhoff-Sobolev parametrix of Klainerman and Rodnianski in \cite{kl_rod:ksp} to systems of tensor wave equations with additional first-order terms.
We also present a different derivation, which better highlights that such representation formulas are supported entirely on past null cones.
This generalization of \cite{kl_rod:ksp} is a key component for extending Klainerman and Rodnianski's breakdown criterion result for Einstein-vacuum spacetimes in \cite{kl_rod:bdc} to Einstein-Maxwell and Einstein-Yang-Mills spacetimes.
\end{abstract}

\maketitle

\section{Introduction}

Let $(M, g)$ denote a $(1+3)$-dimensional time-oriented Lorentzian manifold, with Levi-Civita connection $D$ and Riemann curvature $R$.
Also, fix an integer $n > 0$, along with $n$ integers $\sss{r}{1}, \ldots \sss{r}{n} \geq 0$.
For each $m, c \in \{1, \ldots, n\}$, we let $\sss{\Phi}{m}$ and $\sss{\Psi}{m}$ denote tensor fields on $M$ of rank $\sss{r}{m}$, and we let $\sss{P}{mc}$ denote a tensor field on $M$ of rank $1 + \sss{r}{m} + \sss{r}{c}$.
\footnote{We surround the ``indices" $m$ and $c$ with ``$(\cdot)$" in order to distinguish them from tensorial indices and their associated Einstein summation conventions.}
Assume that these objects satisfy the system
\begin{equation}\label{eq.tensor_wave_system} \sss{\mc{L}}{m} \sss{\Phi}{m}_I = \Box_g \sss{\Phi}{m}_I + \sum_{c = 1}^n \sss{P}{mc}_{\mu I}{}^J D^\mu \sss{\Phi}{c}_J = \sss{\Psi}{m}_I \text{,} \qquad 1 \leq m \leq n \end{equation}
of tensor wave equations on $M$, where $\Box_g = g^{\alpha\beta} D_{\alpha\beta}$ is the covariant wave operator, and where $I$ and $J$ are collections of $\sss{r}{m}$ and $\sss{r}{c}$ indices, respectively.
\footnote{In accordance with Einstein summation conventions, the repeated indices $J$ in \eqref{eq.tensor_wave_system} denote summations over all possible values for $J$.}
Note that if $n = 1$ and $\sss{P}{11}$ vanishes, then \eqref{eq.tensor_wave_system} reduces to the tensorial wave equation
\begin{equation}\label{eq.tensor_wave} \square_g \Phi = \Psi \text{.} \end{equation}

The aim of this paper is to generalize the representation formula of \cite{kl_rod:ksp}, which handled the setting \eqref{eq.tensor_wave}, to also treat systems of the form \eqref{eq.tensor_wave_system}.
In other words, we determine a formula for the $\sss{\Phi}{m}$'s at a point $p \in M$ in terms of the $\sss{\Phi}{m}$'s and $\sss{\Psi}{m}$'s along a portion of the regular past null cone about $p$.
In addition, we present a derivation different in nature from that of \cite{kl_rod:ksp}, which allows us to weaken the assumptions required for this formula to be valid.
Lastly, like in \cite{kl_rod:ksp}, we will extend this representation formula to arbitrary vector bundles.

\subsection{Prior Results}

The model problem for \eqref{eq.tensor_wave_system} is the scalar wave equation
\begin{equation}\label{eq.wave_eq_mink} \square \phi = \psi \text{,} \qquad \phi, \psi \in C^\infty\paren{\R^{1+3}} \end{equation}
on the Minkowski spacetime $\R^{1+3}$.
Suppose $\phi$ solves \eqref{eq.wave_eq_mink}, with initial data
\[ \valat{\phi}_{t = 0} = \alpha_0 \in C^\infty\paren{\R^3} \text{,} \qquad \valat{\partial_t \phi}_{t = 0} = \alpha_1 \in C^\infty\paren{\R^3} \text{.} \]
From standard theory, cf. \cite[Sec. 2.4]{ev:pde}, at a point $(t, x) \in (0, \infty) \times \R^3$, we can express $\phi(t, x)$ explicitly in terms of the initial data $\alpha_0$ and $\alpha_1$.

Recall that we can decompose $\phi = \phi_1 + \phi_2$, where $\phi_1$ satisfies \eqref{eq.wave_eq_mink} but with zero initial data, while $\phi_2$ satisfies the homogeneous equation $\square \phi_2 \equiv 0$ with initial data $\alpha_0$ and $\alpha_1$.
The term $\phi_2$ can be explicitly expressed using Kirchhoff's formula,
\[ \phi_2\paren{t, x} = \frac{1}{4 \pi t^2} \int_{\partial B\paren{x, t}} \brak{t \alpha_1\paren{y} + \alpha_0\paren{y} + \paren{y - x} \cdot \nabla \alpha_0\paren{y}} d\sigma\paren{y} \text{,} \]
where $B(x, t)$ is the ball in $\R^3$ of radius $t$ centered at $x$, and where $d\sigma$ is the standard surface measure for $\partial B(x, t)$.
In addition, from Duhamel's principle,
\[ \phi_1\paren{t, x} = \frac{1}{4 \pi} \int_0^t \int_{\partial B\paren{x, r}} \frac{\psi\paren{y, t - r}}{r} \cdot d\sigma\paren{y} dr = \frac{1}{4 \pi} \int_{B\paren{x, t}} \frac{\psi\paren{y, t - \abs{y - x}}}{\abs{y - x}} \cdot dy \text{.} \]
Note that this can be interpreted as an integral over the null cone segment
\[ N^-_0\paren{t, x} = \brac{\paren{s, y} \in \R \times \R^3 \mid 0 \leq s \leq t \text{, } \abs{y - x} = t - s} \text{.} \]

There exist many geometric analogues for the above model result, which provide possibilities for studying wave equations in other settings.
An important early example is the result of S. Sobolev in \cite{sob:ksp}, in which a first-order parametrix was applied in order to prove well-posedness for second-order linear wave equations with variable coefficients.
Y. Choqu\'et-Bruhat made use of a similar representation formula in her celebrated local existence result for the Einstein-vacuum equations; see \cite{cb:le_einst}.
More recently, Chrusciel and Shatah, in \cite{chru_sh:ym_curv}, used the Hadamard-type parametrix in \cite{fr:wv_eq} in order to extend the classical global existence result of Eardley and Moncrief for the Yang-Mills equations, cf. \cite{ea_mo:l_ymh, ea_mo:g_ymh}, to globally hyperbolic curved spacetimes.

In contrast to the first-order variants of \cite{cb:le_einst, sob:ksp}, which we collectively refer to as ``Kirchhoff-Sobolev-type" parametrices, the Hadamard-type parametrix of \cite{fr:wv_eq} achieved greater precision by making use of infinitely many derivatives of the metric $g$ and requiring the convexity of the domain under consideration.
For instance, since the geometric wave equation no longer obeys the strong Huygens principle, the Hadamard parametrix depends on all points in the causal past $J^-(z)$ of the point $z$ under consideration.
Consequently, geodesic convexity is required to make sense of the formula.
This differs strongly from Minkowski spacetime, in which the fundamental solution in $(1 + 3)$-dimensions is supported entirely on the past null cone.
These smoothness and convexity restrictions, however, can often be undesirable in the context of nonlinear problems in partial differential equations.

We now focus our attention on \cite[Thm. 3.11]{kl_rod:ksp}, which presented a Kirchhoff-Sobolev parametrix directly handling covariant tensorial wave equations on arbitrary $(1 + 3)$-dimensional Lorentzian manifolds.
For convenience, we refer to this representation formula as {\bf KR}.
In particular, {\bf KR}, applied at a point $p \in M$, enjoys the following features:
\begin{itemize}
\item {\bf KR} directly handles the tensorial wave equation \eqref{eq.tensor_wave} in a completely covariant fashion, in particular without reducing to a scalar wave equation.

\item The parametrix is supported entirely on the past null cone of $p$.
Indeed, the formula is expressed as integrals of $\Phi$ and $\Psi$ along this cone.

\item Since the past null cone can degenerate away from $p$, then {\bf KR} is valid only locally near $p$, on a ``regular" portion $\mc{N}^-(p)$ of the cone.
In other words, the validity of {\bf KR} is constrained by the regularity of the null exponential map, a weaker condition than the geodesic convexity assumption of \cite{fr:wv_eq}.

\item The formula {\bf KR} contains ``error terms", expressed as integrals along $\mc{N}^-(p)$ of $\Phi$ along with various geometric quantities on $\mc{N}^-(p)$.
These are a result of the nontrivial geometry of the spacetime.

\item The formula {\bf KR} depends only on quantities defined on $\mc{N}^-(p)$, i.e., it is independent of extensions of any quantities to neighborhoods of $\mc{N}^-(p)$.

\item The formula {\bf KR} can be systematically generalized to covariant wave equations on arbitrary vector bundles over $M$ with a bundle metric and a compatible connection; see \cite[Thm. 4.1]{kl_rod:ksp}.
\end{itemize}

For instance, \cite[Thm. 4.1]{kl_rod:ksp} can be used to reprove global existence for the Yang-Mills equations in $\R^{3+1}$ in a similar manner as in \cite{ea_mo:g_ymh}; see \cite[Sec. 5]{kl_rod:ksp}.
In addition, since {\bf KR} uses fully tensorial and covariant methods, then the above can be achieved without reference to Cronstr\"om gauges.
This ``gauge-invariant" method can likely be extended to treat curved spacetimes as well.

Moreover, {\bf KR} was applied in the proof of the breakdown criterion for the Einstein-vacuum equations in \cite{kl_rod:bdc}.
\footnote{The entirety of the proof of the breakdown/continuation argument stated in \cite{kl_rod:bdc}, however, spans multiple papers, including \cite{kl_rod:cg, kl_rod:glp, kl_rod:stt, kl_rod:ksp, kl_rod:rin, kl_rod:bdc, wang:cg, wang:cgp}.}
The main result of \cite{kl_rod:bdc} is the following:

\begin{theorem}\label{thm.bdc_vacuum}
Suppose $(M, g)$ is an existing Einstein-vacuum spacetime, given as a constant mean curvature foliation
\[ M = \bigcup_{t_0 < \tau < t_1} \Sigma_\tau \text{,} \qquad t_0 < t_1 < 0 \text{,} \]
where each $\Sigma_\tau$ is a compact spacelike hypersurface of $M$ satisfying the constant mean curvature condition $\trace k \equiv \tau$, and where $k$ denotes the second fundamental form of $\Sigma_\tau$ in $M$.
In addition, assume the following criterion:
\[ \nabs{k}_{L^\infty\paren{M}} + \nabs{\nabla \paren{\log n}}_{L^\infty\paren{M}} < \infty \text{.} \]
Then, $(M, g)$ can be extended as a CMC foliation to some time $t_1 + \epsilon$.
\end{theorem}

In particular, {\bf KR} was utilized to derive $L^\infty$-bounds for the curvature, which was shown to satisfy a covariant tensorial wave equation; see \cite[Sec. 5]{kl_rod:bdc}.
These $L^\infty$-bounds were essential for obtaining uniform energy estimates along the timeslices of the spacetime.
Using local existence theory and various elliptic estimates, one could then derive Theorem \ref{thm.bdc_vacuum}.

\subsection{The Generalized Formula}

The parametrix {\bf KR}, however, is insufficient for handling the analogous breakdown problem for the Einstein-Maxwell (and similarly, the Einstein-Yang-Mills) setting.
The main result, discussed in \cite{shao:bdc_nv}, is the following:

\begin{theorem}\label{thm.bdc_maxwell}
Suppose $(M, g, F)$ is an existing Einstein-Maxwell spacetime, where $F$ is a $2$-form in $M$ representing the Maxwell field.
Assume $M$ is given as a CMC foliation as in Theorem \ref{thm.bdc_vacuum}, and assume the following criterion:
\[ \nabs{k}_{L^\infty\paren{M}} + \nabs{\nabla \paren{\log n}}_{L^\infty\paren{M}} + \nabs{F}_{L^\infty\paren{M}} < \infty \text{.} \]
Then, $(M, g, F)$ can be extended as a CMC foliation.
\end{theorem}

The failure of {\bf KR} in the proof of Theorem \ref{thm.bdc_maxwell} is due to the presence of nontrivial first-order terms in the wave equations for the curvature and the Maxwell tensors.
More specifically, if we apply {\bf KR}, then we obtain null cone integrals containing ``bad" components of the derivative of the curvature, which we cannot control a priori using standard local energy estimates.
We also obtain equally untreatable ``bad" components for the Maxwell tensor.

In order to obtain Theorem \ref{thm.bdc_maxwell}, we must deal with the aforementioned first-order terms.
To achieve this, we generalize {\bf KR} to handle a system of tensor wave equations of the form \eqref{eq.tensor_wave_system}.
In particular, we must handle the first-order portion of \eqref{eq.tensor_wave_system} differently.
As we shall see later, the main ``trick" will be to eliminate those terms containing derivatives of the $\sss{\Phi}{m}$'s which are transverse to the null cone.
This is intuitively unsurprising, since such derivatives, in both flat and curved spacetimes, are generally the most difficult to control using local energy estimates.

Before delving fully into technical details, we first compare the approach of this generalized representation formula to that of its predecessor {\bf KR}:
\begin{itemize}
\item Similar to developments in \cite{fr:wv_eq}, we handle the contributions of the first-order terms by suitably altering the transport equation in \cite{kl_rod:ksp}.
We add corresponding terms to these transport equations, which include the $\sss{P}{mc}$'s.

\item Due to coupling, all the equations in \eqref{eq.tensor_wave_system} must be handled concurrently.
Consequently, the transport equation of \cite{kl_rod:ksp} becomes a coupled system of $n$ transport equations.

\item Like for {\bf KR}, the parametrix of this paper will be supported on regular past null cones and will depend entirely on quantities defined on these cones.
\end{itemize}

We opt for a different approach to the derivation of our formula than for that of {\bf KR}.
The main differences in principle between the two proofs are as follows:
\begin{itemize}
\item A major component in deriving {\bf KR} is the optical function, whose level sets are null cones.
Here, we avoid any reference to such a function.
As a result, we can eliminate some extraneous geometric assumptions needed in \cite{kl_rod:ksp}.

\item While \cite{kl_rod:ksp} makes heavy use of distributions on manifolds in an informal fashion, we take a more technically rigorous path and remain at the level of the calculus operations represented by such distributions.
In particular, rather than dealing with derivatives of the $\delta$-distribution like in \cite{kl_rod:ksp}, we deal with corresponding integrations by parts.

\item In our derivation, we integrate by parts only the derivatives in directions tangent to the null cone, while in \cite{kl_rod:ksp}, integration by parts is applied liberally to derivatives in all directions.
The latter approach results in terms depending on quantities off the null cone, which must then be cancelled through additional meticulous integrations by parts.
By exercising further discretion, we can avoid such troubles altogether.

\item In this paper, we will make explicit the dependence of the parametrix on the initial data.
\footnote{By ``initial data", we mean the values of the $\sss{\Phi}{m}$'s and $D \sss{\Phi}{m}$'s on a spherical cross-section of the past null cone.  These correspond $\alpha_0$ and $\alpha_1$ in the scalar Minkowski model problem.}
This is in contrast to \cite{kl_rod:ksp}, which focused only on the case of vanishing initial data.
\end{itemize}

As mentioned before, we obtain subtle improvements over the assumptions required for {\bf KR} to be valid.
Assumptions {\bf A1} and {\bf A2} of \cite[Sec. 2.1]{kl_rod:ksp}, assumed when deriving {\bf KR} at a point $p \in M$, postulated that certain local hyperbolicity and null regularity conditions hold for all points in some neighborhood of $p$.
In contrast, in our derivation, since we avoid the optical function and work exclusively on the null cone, we require only the null regularity condition for $p$.

We also previously noted that both the representation formula of this paper and {\bf KR} depend only on quantities defined on the null cone.
Since the derivation of {\bf KR} introduces a multitude of terms which depend on quantities off the cone, that the final result does not depend on quantities off the cone appears as a ``miracle" resulting from numerous cancellations.
In contrast, in our derivation, we automatically have at each step of the proof that every term is dependent only on quantities defined on the null cone.
Therefore, our derivation highlights this property of the parametrix as a natural rather than surprising consequence.

\subsection{The Main Result}

Now, we state an abridged schematic version of the main theorem of this paper.
The precise statement will be deferred until Theorem \ref{thm.ksp_gen}, after we develop the background and notations needed to fully describe the parametrix.

\begin{theorem}\label{thm.ksp_gen_pre}
Consider the system \eqref{eq.tensor_wave_system} on $(M, g)$, let $\mc{N}$ denote a ``regular" segment of the past null cone about $p \in M$, and let $\sss{J}{1}, \ldots, \sss{J}{n}$ be tensors at $p$ of the same ranks as $\sss{\Phi}{1}, \ldots, \sss{\Phi}{n}$, respectively.
Then, we have the schematic formula
\[ \sum_{m = 1}^n \sss{J}{m} \cdot \valat{\sss{\Phi}{m}}_p = - \sum_{m = 1}^n \int_{\mc{N}} \sss{A}{m} \cdot \sss{\Psi}{m} + \paren{\text{error terms}} + \paren{\text{initial data}} \text{,} \]
where:
\begin{itemize}
\item The $\sss{A}{m}$'s are tensor fields on $\mc{N}$ which satisfy a system of transport equations depending on the $\sss{J}{m}$'s and on the geometry of $\mc{N}$.

\item The error terms can be expressed as integrals of quantities on $\mc{N}$.

\item The initial data contribution can be expressed as integrals over the lower boundary of $\mc{N}$ (i.e., a spherical cross-section of the past null cone).
\end{itemize}
\end{theorem}

In addition, in the final section, we will discuss extensions of the parametrix to covariant wave equations on sections of vector bundles, similar to \cite[Thm. 4.1]{kl_rod:ksp}, but once again with nontrivial first-order terms.
We also discuss the main theorem in the context of this extended framework, and we discuss how this extension can be applied to treat the Einstein-Yang-Mills breakdown problem.

\subsection*{Acknowledgements}

The author would like to thank Professor Sergiu Klainerman for his insights regarding this topic, and for hours of helpful discussions.

\section{Geometry of Regular Null Cones}

Fix $p \in M$, and consider the \emph{null exponential map} $\ul{\exp}_p$ about $p$, defined as the restriction of the exponential map about $p$ to the past null cone $\mf{N}$ of the tangent space $T_p M$.
\footnote{Here, we assume $\mf{N}$ does not include the origin of $T_p M$.}
Define the past null cone $N^-(p)$ of $p$ to be the image of $\ul{\exp}_p$.
Note that $N^-(p)$ is ruled by the past inextendible null geodesics beginning at $p$.

We call a neighborhood $U$ in $\mf{N}$ of the origin \emph{star-shaped} iff for each $q \in U$, the open segment between the origin and $q$ is also contained in $U$.
Fix such a star-shaped $U \subseteq \mf{N}$, and assume $\ul{\exp}_p$ is also a global diffeomorphism on $U$.
From now on, we shall denote the image $\ul{\exp}_p(U)$ by $\mc{N}^-(p)$, which we refer to as the \emph{regular past null cone}.
\footnote{In other words, we can think of $\mc{N}^-(p)$ as a ``past null convex" set.}
In particular, $\mc{N}^-(p)$ is a smooth null hypersurface of $M$.

\begin{remark}
Recall that since one can find convex subsets about any point, then one can always construct such a set $\mc{N}^-(p)$.
\end{remark}

\begin{remark}
For nonpathological spacetimes, we can systematically construct such $\mc{N}^-(p)$ via global considerations.
For instance, in the globally hyperbolic setting, we can determine $\mc{N}^-(p)$ from the past null boundary $J^-(p) - I^-(p)$, where $I^-(p)$ and $J^-(p)$ are the chronological and causal pasts of $p$, respectively; see \cite{haw_el:gr, on:srg}.
In practice, this is the situation adopted in most applications.
\end{remark}

We remark here that a loss of such regularity occurs at what are called \emph{terminal points} of $N^-(p)$.
At such a point $z$, one of the following scenarios hold:
\begin{itemize}
\item Two distinct past null geodesics from $p$ intersect at $z$.
In other words, the map $\ul{\exp}_p$ fails to be one-to-one at $z$.

\item The pair $p$ and $z$ are null conjugate points.
In other words, the map $\ul{\exp}_p$ fails to be nonsingular at $z$.
\end{itemize}

From now on, we shall deal only with the regular portion $\mc{N}^-(p)$, on which we have a smooth structure.
Furthermore, since tangent null vectors in $\mc{N}^-(p)$ have vanishing Lorentzian ``length" and are orthogonal to $\mc{N}^-(p)$, they cannot be normalized without introducing vectors transversal to $\mc{N}^-(p)$.
Consequently, in our treatment, we will require an additional choice of a future timelike unit vector $\mf{t} \in T_pM$.
Here, we fix such a $\mf{t}$ for the remainder of this paper.

\subsection{Null Generators}

We define the \emph{null generators} of $\mc{N}^-(p)$ to be the inextendible past null geodesics $\gamma$ on $M$ which satisfy $\gamma(0) = p$ and $g(\gamma^\prime(0), \mf{t}) = 1$.
We can smoothly parametrize these generators by $\Sph^2$ using the following process.
If we choose an orthonormal basis $e_0, \ldots, e_3$ of $T_pM$, with $e_0 = \mf{t}$, then we can identify each $\omega \in \Sph^2$ with the null generator $\gamma_\omega$ satisfying $\gamma_\omega^\prime(0) = -e_0 + \omega^k e_k$.
For convenience, we assume such a parametrization of the null generators of $\mc{N}^-(p)$, and we denote by $\gamma_\omega$, $\omega \in \Sph^2$, the null generator corresponding to $\omega$.

\begin{remark}
The objects on $\mc{N}^-(p)$ that we will discuss can of course be defined independently of any parametrization of the null generators.
However, for ease of notation, we will work explicitly with $\Sph^2$.
\end{remark}

In addition, we define the vector field $L$ on $\mc{N}^-(p)$ to be the tangent vector fields of the null generators, i.e., we define $L|_{\gamma_\omega(v)} = \gamma_\omega^\prime(v)$ for any $\omega \in \Sph^2$.
Note in particular that $L$ is a geodesic vector field.

\subsection{Spherical Foliating Functions}

We say that a function $f \in C^\infty(\mc{N}^-(p))$ is a \emph{foliating function} of $\mc{N}^-(p)$ iff $f$ satisfies the following conditions:
\begin{itemize}
\item For every $\omega \in \Sph^2$,
\[ \lim_{v \searrow 0} \valat{f}_{\gamma_\omega\paren{v}} = 0 \text{.} \]

\item The function
\begin{equation}\label{eq.null_lapse} \vartheta = \vartheta_f = \paren{L f}^{-1} \in C^\infty\paren{\mc{N}^-\paren{p}} \end{equation}
is everywhere positive on $\mc{N}^-(p)$.
In other words, $f$ is strictly increasing along every null generator of $\mc{N}^-(p)$.

\item There is a constant $\vartheta_0 > 0$ such that for all $\omega \in \Sph^2$,
\[ \lim_{v \searrow 0} \valat{\vartheta}_{\gamma_\omega\paren{v}} = \vartheta_0 \text{.} \]

\item For every $\omega \in \Sph^2$, the limit
\[ \lim_{s \searrow 0} \valat{L \vartheta}_{\gamma_\omega\paren{s}} \]
exists and is finite.
\end{itemize}
The function $\vartheta$ defined in \eqref{eq.null_lapse} is called the \emph{null lapse} of $f$.

\begin{remark}
We can in fact weaken the third condition in the above definition so that $\vartheta$ converges to different limits along various null generators.
This generalization, however, is currently without practical applications and complicates various initial limit computations.
Therefore, we restrict ourselves to the simpler definition.
\end{remark}

In the remainder of this paper, we will let $f$ denote an arbitrary ``default" foliating function for $\mc{N}^-(p)$, and we let $\vartheta = \vartheta_f$.
The positivity of $\vartheta$ implies $df$ is everywhere nonvanishing, so the level sets of $f$, denoted
\[ \mc{S}_v = \mc{S}^f_v = \brac{q \in \mc{N}^-\paren{p} \mid f\paren{q} = v} \text{,} \qquad v > 0 \text{,} \]
form a family of hypersurfaces of $\mc{N}^-(p)$.
Since $L$ represents the unique null direction tangent to $\mc{N}^-(p)$, and the positivity of $\vartheta$ implies that $L$ is transverse to each $\mc{S}_v$, we can conclude that each $\mc{S}_v$ is spacelike, i.e., Riemannian.
We shall adopt the following notational conventions: we let $\lambda$ and $\nasla$ denote the induced metrics and the Levi-Civita connections on the $\mc{S}_v$'s, respectively.

Define the quantity $\mf{i}(p) = \mf{i}_f(p)$ to be the supremum of all $v > 0$ such that for every $\omega \in \Sph^2$, there exists a point $q$ on both $\gamma_\omega$ and $\mc{N}^-(p)$ such that $f(q) = v$.
This can be interpreted as $N^-(p)$ remaining regular up to $f$-value $\mf{i}(p)$.
\footnote{If $\mc{N}^-(p)$ is defined in terms of the global geometry of $M$, as was discussed before, then $\mf{i}(p)$ corresponds to the ``null injectivity radius" of $\mc{N}^-(p)$ with respect to $f$.}
Note that the definition of $\mf{i}(p)$ and the positivity of $\vartheta$ imply that $\mc{S}_v$ is diffeomorphic to $\Sph^2$ whenever $0 < v < \mf{i}(p)$.

In addition, for any $0 < v_1 \leq \mf{i}(p)$, we define the null cone segment
\[ \mc{N}^-\paren{p; v_1} = \mc{N}^-_f\paren{p; v_1} = \brac{q \in \mc{N}^-\paren{p} \mid f\paren{q} < v_1} \text{,} \]
while for any $0 < v_1 < v_2 \leq \mf{i}(p)$, we define
\[ \mc{N}^-\paren{p; v_1, v_2} = \mc{N}^-_f\paren{p; v_1, v_2} = \brac{q \in \mc{N}^-\paren{p} \mid v_1 < f\paren{q} < v_2} \text{.} \]

The most natural example of a foliating function is the \emph{affine parameter} $s$ on $\mc{N}^-(p)$, given by $s(\gamma_\omega(v)) = v$ for each $\omega \in \Sph^2$.
Note in particular that $s$ converges to $0$ at $p$ and $\vartheta_s \equiv 1$ everywhere.
The foliation of $\mc{N}^-(p)$ by the $\mc{S}^s_v$'s is called the \emph{geodesic foliation}.

\begin{remark}
Note that $s$ depends on the normalization $\mf{t}$.
\end{remark}

If the spacetime $M$ is foliated by a global time function $t$, then we have another natural foliating function $t_p$ for $\mc{N}^-(p)$, given by $t_p(q) = t(p) - t(q)$.
This is the foliating function used in the breakdown results of \cite{kl_rod:bdc, shao:bdc_nv}.

\subsection{Tensors}

Here, we list the notations we will use to describe various tensor fields on $\mc{N}^-(p)$.
Again, we assume an arbitrary foliating function $f$ for $\mc{N}^-(p)$.
We begin with the following objects:
\begin{itemize}
\item A tensor $w$ at some $q \in \mc{S}_v$ is said to be \emph{horizontal} iff $w$ is tangent to $\mc{S}_v$.

\item We denote by $\ul{T}^k \mc{N}^-(p)$ the \emph{horizontal bundle} over $\mc{N}^-(p)$ of all horizontal tensors of total rank $k$ at every $q \in \mc{N}^-(p)$.

\item Similarly, we denote by $\ol{T}^l \mc{N}^-(p)$ the \emph{extrinsic bundle} over $\mc{N}^-(p)$ of all tensors in $M$ of total rank $l$ at every $q \in \mc{N}^-(p)$.
\end{itemize}

For a vector bundle $\mc{V}$, we let $\Gamma \mc{V}$ denote the space of all smooth sections of $\mc{V}$.
By this formalism, then $\Gamma \ul{T}^k \mc{N}^-(p)$ and $\Gamma \ol{T}^l \mc{N}^-(p)$ denote the spaces of all \emph{horizontal tensor fields} of rank $k$ and \emph{extrinsic tensor fields} of rank $l$ on $\mc{N}^-(p)$, respectively.
In other words, a horizontal tensor field $A \in \Gamma \ul{T}^k \mc{N}^-(p)$ smoothly maps each $q \in \mc{N}^-(p)$ to a horizontal tensor of rank $k$ at $q$; extrinsic tensor fields can be characterized analogously.
For example, the restrictions to $\mc{N}^-(p)$ of the spacetime metric $g$ and the curvature $R$ can be treated as elements of $\Gamma \ol{T}^2 \mc{N}^-(p)$ and $\Gamma \ol{T}^4 \mc{N}^-(p)$, while the induced metrics $\lambda$ on the $\mc{S}_v$'s can be treated as an element of $\Gamma \ul{T}^2 \mc{N}^-(p)$.
Lastly, note that $\Gamma \ul{T}^0 \mc{N}^-(p) = \Gamma \ol{T}^0 \mc{N}^-(p) = C^\infty(\mc{N}^-(p))$.

Next, we define the \emph{mixed bundles} over $\mc{N}^-(p)$ to be the tensor product bundles
\[ \ul{T}^k \ol{T}^l \mc{N}^-\paren{p} = \ul{T}^k \mc{N}^-\paren{p} \otimes \ol{T}^l \mc{N}^-\paren{p} \text{.} \]
Similarly, we will call a section $A \in \Gamma \ul{T}^k \ol{T}^l \mc{N}^-(p)$ a \emph{mixed tensor field} on $\mc{N}^-(p)$.
By the duality formulation, we can consider such a field $A$ as a bilinear map
\[ A: \Gamma \ul{T}^k \mc{N}^-\paren{p} \times \Gamma \ol{T}^l \mc{N}^-\paren{p} \rightarrow C^\infty\paren{\mc{N}^-\paren{p}} \text{,} \]
or as the corresponding $C^\infty(\mc{N}^-(p))$-valued multilinear map on $k$ horizontal vector fields and $l$ extrinsic vector fields.

In index notation, we adopt the following conventions:
\begin{itemize}
\item Horizontal indices will be denoted using Latin letters and will take values between $1$ and $2$, inclusive.

\item Extrinsic indices will be denoted using Greek letters and will take values between $1$ and $4$, inclusive.

\item Collections of extrinsic indices will be denoted using capital Latin letters.
\end{itemize}

\subsection{Covariant Differentiation}

Recall that the Levi-Civita connection $D$ on $M$ induces a connection $\ol{D}$ on the extrinsic bundles $\ol{T}^l \mc{N}^-(p)$.
For any $A \in \Gamma \ol{T}^l \mc{N}^-(p)$ and a vector field $X$ on $\mc{N}^-(p)$, we can arbitrarily extend $A$ to a neighborhood of $\mc{N}^-(p)$ and define $\ol{D}_X A = D_X A$.
It is easy to see that this definition is independent of the chosen extension of $A$.
Moreover, it is clear from definition that $\ol{D}_X g \equiv 0$ for any vector field $X$ on $\mc{N}^-(p)$.
\footnote{Technically, by $\ol{D}_X g$, we mean $\ol{D}_X$ acting on the restriction of $g$ to $\mc{N}^-(p)$.}

The Levi-Civita connections $\nasla$ on the $\mc{S}_v$'s naturally aggregate to define a connection $\nasla$ on the horizontal bundles $\Gamma \ul{T}^k \mc{N}^-(p)$.
Let $X$ be a vector field on $\mc{N}^-(p)$:
\begin{itemize}
\item For a scalars $f \in C^\infty(\mc{N}^-(p))$, we define $\nasla_X f = Xf$, as usual.

\item For a horizontal vector field $Y$, we define $\nasla_X Y$ to be the projection onto the $\mc{S}_v$'s of $\ol{D}_X Y$.
Note that if $X$ is horizontal, then $\nasla_X Y$ is precisely the corresponding covariant derivative with respect to the $\mc{S}_v$'s.

\item For a fully covariant $A \in \Gamma \ul{T}^k \mc{N}^-(p)$, then $\nasla_X A \in \Gamma \ul{T}^k \mc{N}^-(p)$ is naturally given as follows: for horizontal vector fields $Y_1, \ldots, Y_k$,
\begin{align*}
\nasla_X A\paren{Y_1, \ldots, Y_k} &= X\brak{A\paren{Y_1, \ldots, Y_k}} - A\paren{\nasla_X Y_1, Y_2, \ldots, Y_k} \\
&\qquad - \ldots - A\paren{Y_1, \ldots, \nasla_X Y_k} \text{.}
\end{align*}
\end{itemize}
This definition of $\nasla$ generalizes the usual $\nasla$-covariant derivative on the $\mc{S}_v$'s to also include the $L$-direction.
Note also that $\nasla_X \lambda \equiv 0$ for any vector field $X$ on $\mc{N}^-(p)$, i.e., $\nasla$ remains compatible with the horizontal metrics.

We can canonically combine the connections $\nasla$ and $\ol{D}$ on the horizontal and extrinsic bundles to obtain generalized connections $\ol{\nasla}$ on the mixed bundles.
The basic idea is to have $\ol{\nasla}$ behave ``like $\nasla$ on the horizontal components" and ``like $\ol{D}$ on the extrinsic components".
More specifically, if $A \in \Gamma \ul{T}^k \ol{T}^l \mc{N}^-(p)$ is a fully covariant mixed tensor field, $X$ is a vector field on $\mc{N}^-(p)$, $Y_1, \ldots, Y_k \in \Gamma \ul{T}^1 \mc{N}^-(p)$ are horizontal vector fields, and $Z_1, \ldots, Z_l \in \Gamma \ol{T}^1 \mc{N}^-(p)$ are extrinsic vector fields, then we define $\ol{\nasla}_X A \in \Gamma \ul{T}^k \ol{T}^l \mc{N}^-(p)$, interpreted as a multilinear map, by assigning to the expression $\ol{\nasla}_X A(Y_1, \ldots, Y_k; Z_1, \ldots, Z_l)$ the following value:
\begin{align*}
&X\brak{A\paren{Y_1, \ldots, Y_k; Z_1, \ldots, Z_l}} - A\paren{\nasla_X Y_1, Y_2, \ldots, Y_k; Z_1, \ldots Z_l} - \ldots \\
&\qquad - A\paren{Y_1, \ldots, Y_{k-1}, \nasla_X Y_k; Z_1, \ldots, Z_l} - A\paren{Y_1, \ldots, Y_k; \ol{D}_X Z_1, Z_2, \ldots Z_l} \\
&\qquad - \ldots - A\paren{Y_1, \ldots, Y_k; Z_1, \ldots, Z_{l-1}, \ol{D}_X Z_l} \text{.}
\end{align*}

The main observations concerning the above construction are the following:
\begin{itemize}
\item The mixed covariant derivatives satisfy Leibniz rules similar to $\nasla$ and $\ol{D}$-covariant derivatives.

\item For any vector field $X$ on $\mc{N}^-(p)$, both $\ol{\nasla}_X \lambda \equiv 0$ and $\ol{\nasla}_X g \equiv 0$.
\end{itemize}
The primary consequence of these observations is that the same integrations by parts operations can be justified for $\ol{\nasla}$-derivatives of mixed tensor fields as for $\nasla$-derivatives on horizontal tensor fields.
This property was implicitly used in \cite{kl_rod:ksp, kl_rod:bdc}.

Finally, we will make use of the following notations:
\begin{itemize}
\item For any $A \in \Gamma \ul{T}^k \mc{N}^-(p)$, we let $\nasla A \in \Gamma \ul{T}^{k+1} \mc{N}^-(p)$ denote the horizontal tensor field mapping a horizontal vector field $X$ to $\nasla_X A$.
This is precisely the covariant differentials of $A$ on the $\mc{S}_v$'s.

\item Similarly, for any $A \in \Gamma \ul{T}^k \ol{T}^l \mc{N}^-(p)$, we let $\ol{\nasla} A \in \Gamma \ul{T}^{k+1} \ol{T}^l \mc{N}^-(p)$ denote the mixed tensor field mapping a horizontal vector field $X$ to $\ol{\nasla}_X A$.

\item We also define the horizontal and mixed Laplacians in the usual way:
\[ \lasl = \lambda^{ab} \nasla_{ab} \text{,} \qquad \ol{\lasl} = \lambda^{ab} \ol{\nasla}_{ab} \text{.} \]

\item Consider also the operators $\nasla_f = \vartheta \nasla_L$ and $\ol{\nasla}_f = \vartheta \ol{\nasla}_L$.
These are horizontal and mixed covariant derivatives in the tangent null direction, subject to the normalizations $\nasla_f f = \ol{\nasla}_f f \equiv 1$.
\end{itemize}
For further details involving the above constructions, see \cite[Sec. 1.2]{shao:bdc_nv}.

\subsection{Parametrizations}

We can parametrize $\mc{N}^-(p)$ by the foliating function $f$ and a value $\omega \in \Sph^2$.
For $0 < v < \mf{i}(p)$ and $\omega \in \Sph^2$, we can identify $(v, \omega)$ with the unique point $q$ on the corresponding null generator $\gamma_\omega$ with $f(q) = v$.
As a result, we can naturally treat any $\phi \in C^\infty(\mc{N}^-(p))$ as a function of $f$ and $\omega$.
For any such $\phi$, we denote by $\phi|_{(v, \omega)}$ the value of $\phi$ at the point $q$ corresponding to the parameters $(v, \omega)$.
We will freely use this $(v, \omega)$-notation throughout future sections without further elaboration.

\subsection{Null Frames}

In general, null frames are local frames $\hat{l}, \hat{m}, e_1, e_2$ which satisfy
\begin{align*}
g\paren{\hat{l}, \hat{l}} = g\paren{\hat{m}, \hat{m}} \equiv 0 \text{, } &\qquad g\paren{\hat{l}, \hat{m}} \equiv -2 \text{,} \\
g\paren{\hat{l}, e_a} = g\paren{\hat{m}, e_a} \equiv 0 \text{, } &\qquad g\paren{e_a, e_b} = \delta_{ab} \text{,}
\end{align*}
Here, we define null frames which are adapted to the $f$-foliation of $\mc{N}^-(p)$.

Each point of $\mc{S}_v$ is normal to exactly two null directions, one of which is represented by $L$.
We define $\ul{L}$, called the \emph{conjugate null vector field}, to be the vector field in the other normal null direction, subject to the normalization $g(L, \ul{L}) \equiv -2$.
Next, we append to $L$ and $\ul{L}$ a local orthonormal frame $e_1, e_2$ on the $\mc{S}_v$'s.
Then, $\{L, \ul{L}, e_1, e_2\}$ defines a natural null frame for $\mc{N}^-(p)$.

In this paper, we will index only with respect to adapted null frames.
\begin{itemize}
\item Horizontal indices $1, 2$ correspond to the directions $e_1$ and $e_2$.

\item $\ul{L}$ corresponds to the index $3$, while $L$ corresponds to the index $4$.
\end{itemize}

\subsection{Ricci Coefficients}

We will make use of the following connection quantities:
\begin{itemize}
\item Define the null second fundamental forms $\chi, \ul{\chi} \in \Gamma \ul{T}^2 \mc{N}^-(p)$ by
\[ \chi\paren{X, Y} = g\paren{\ol{D}_X L, Y} \text{,} \qquad \ul{\chi}\paren{X, Y} = g\paren{\ol{D}_X \ul{L}, Y} \]
for any horizontal vector fields $X$ and $Y$ on $\mc{N}^-(p)$.
Note that $\chi$ and $\ul{\chi}$ are symmetric, since both $L$ and $\ul{L}$ are normal to the $\mc{S}_v$'s.

\item We often decompose $\chi$ into its trace and traceless parts:
\[ \trace \chi = \lambda^{ab} \chi_{ab} \text{,} \qquad \hat{\chi} = \chi - \frac{1}{2} \paren{\trace \chi} \lambda \text{.} \]
We will also use an analogous decomposition for $\ul{\chi}$.

\item Define $\zeta, \ul{\eta} \in \Gamma \ul{T}^1 \mc{N}^-(p)$ by
\[ \zeta\paren{X} = \frac{1}{2} g\paren{\ol{D}_X L, \ul{L}} \text{,} \qquad \ul{\eta}\paren{X} = \frac{1}{2} g\paren{X, \ol{D}_L \ul{L}} \]
for any horizontal vector field $X$ on $\mc{N}^-(p)$.
\end{itemize}
The quantities $\trace \chi$, $\hat{\chi}$, and $\zeta$ are called the expansion, shear, and torsion of $\mc{N}^-(p)$ (with respect to the $f$-foliation), respectively.

In addition, we have the following relation between $\zeta$ and $\ul{\eta}$:
\begin{equation}\label{eq.null_torsion} \ul{\eta} = -\zeta + \nasla \paren{\log \vartheta} \text{.} \end{equation}
For a proof, see \cite[Prop. 2.7]{kl_rod:cg}.

Given a null frame $L, \ul{L}, e_1, e_2$, a vector field $Z \in \Gamma \ol{T}^1 \mc{N}^-(p)$ is decomposed
\[ Z = -\frac{1}{2} g\paren{Z, \ul{L}} L - \frac{1}{2} g\paren{Z, L} \ul{L} + \sum_{a = 1}^2 g\paren{Z, e_a} e_a \text{.} \]
Using this identity, we can decompose covariant derivatives along $\mc{N}^-(p)$:
\begin{align}
\label{eq.D_nf} \ol{D}_a e_b = \nasla_a e_b + \frac{1}{2} \ul{\chi}_{ab} L + \frac{1}{2} \chi_{ab} \ul{L} \text{,} &\qquad \ol{D}_a L = \chi_a{}^b e_b - \zeta_a L \text{,} \\
\notag \ol{D}_a \ul{L} = \ul{\chi}_a{}^b e_b + \zeta_a \ul{L} \text{,} &\qquad \ol{D}_L e_a = \nasla_L e_a + \ul{\eta}_a L \text{,} \\
\notag \ol{D}_L L \equiv 0 \text{,} &\qquad \ol{D}_L \ul{L} = 2 \ul{\eta}^a e_a \text{.}
\end{align}

Lastly, we define the \emph{mass aspect function} $\mu \in C^\infty(\mc{N}^-(p))$ by
\begin{equation}\label{eq.maf} \mu = \nasla^a \zeta_a - \frac{1}{2} \hat{\chi}^{ab} \hat{\ul{\chi}}_{ab} + \abs{\zeta}^2 + \frac{1}{4} R_{4343} - \frac{1}{2} R_{43} \text{,} \end{equation}
and we note the following transport equation satisfied by $\trace \ul{\chi}$:
\begin{equation}\label{eq.str_ev_chib_tr} \ol{\nasla}_L \paren{\trace \ul{\chi}} = 2 \nasla^a \ul{\eta}_a + 2 \abs{\ul{\eta}}^2 - \frac{1}{2} \paren{\trace \chi} \paren{\trace \ul{\chi}} - \hat{\chi}^{ab} \hat{\ul{\chi}}_{ab} + \frac{1}{2} R_{4343} - R_{43} \text{.} \end{equation}
For details on \eqref{eq.str_ev_chib_tr}, see \cite{chr_kl:stb_mink} or \cite[Sec. 3.2]{shao:bdc_nv}.

\subsection{Integration}

Since $\mc{N}^-(p)$ is null, we have no volume form on $\mc{N}^-(p)$ with respect to which we can integrate scalar functions.
However, we can still give a natural definition for integrals of functions over $\mc{N}^-(p)$, as long as we have a fixed normalization for $\mc{N}^-(p)$.
Indeed, we define this integral by
\begin{equation}\label{eq.pnc_int} \int_{\mc{N}^-(p)} \phi = \int_0^\infty \paren{\int_{\mc{S}^s_v} \phi} dv \end{equation}
for any $\phi \in C^\infty(\mc{N}^-(p))$ for which the right hand side is well-defined, where the $\mc{S}^s_v$'s are the level sets of $s$ on $\mc{N}^-(p)$.
We can similarly define integrals over any open subset of $\mc{N}^-(p)$, in particular the segments $\mc{N}^-(p; v_2)$ and $\mc{N}^-(p; v_1, v_2)$.

By a change of variables, we can restate \eqref{eq.pnc_int} in terms of any foliating function $f$:

\begin{proposition}\label{thm.pnc_int}
For any foliating function $f$ of $\mc{N}^-(p)$, we have
\[ \int_{\mc{N}^-\paren{p}} \phi = \int_0^\infty \paren{\int_{\mc{S}_v} \vartheta \cdot \phi} dv \]
for any integrable $\phi \in C^\infty(\mc{N}^-(p))$.
\end{proposition}

We will also need the following derivative formula.
Here, the integrals over $\mc{S}_v$ are the usual integrals over Riemannian manifolds.

\begin{proposition}\label{thm.evolution_pnc_int}
Let $\phi \in C^\infty(\mc{N}^-(p))$.
Then, for $0 < v_0 < \mf{i}(p)$,
\[ \valat{\frac{d}{dv} \int_{\mc{S}_v} \phi}_{v = v_0} = \int_{\mc{S}_{v_0}} \brak{\nasla_f \phi + \vartheta \paren{\trace \chi} \phi} \text{.} \]
\end{proposition}

For proofs of the above propositions, see \cite[Sec. 3.4]{shao:bdc_nv} and \cite{kl_rod:cg}.

\begin{remark}
We can justify the definition \eqref{eq.pnc_int} using energy estimate considerations.
Suppose $V$ is a region in $M$, with a portion of its boundary given by a neighborhood $V^\prime$ in $\mc{N}^-(p)$.
Then, if $X$ is a $1$-form on $M$, and we integrate $D^\alpha X_\alpha$ over $V$ and apply the divergence theorem, then the boundary integral corresponding to $V^\prime$ is precisely given by
\begin{equation}\label{eq.pnc_int_boundary} \int_{V^\prime} g\paren{X, L} \text{.} \end{equation}
This formula is fundamental to the local energy estimates applied in \cite{kl_rod:bdc, shao:bdc_nv}.
\end{remark}

\begin{remark}
Note that the definition \eqref{eq.pnc_int} depends on the normalization $\mf{t}$ chosen for $\mc{N}^-(p)$, while the argument in the previous remark establishes that the expression \eqref{eq.pnc_int_boundary} is in fact independent of $\mf{t}$.
\end{remark}

\subsection{Initial Values}

Another important set of properties we shall need in our analysis concerns the ``initial values" of many of the objects we have defined above, that is, we wish to compute their limits at $p$ along the null generators of $\mc{N}^-(p)$.
We list the general results for arbitrary foliations below.
Their proofs involve applications of convex geometry among other technicalities, hence we omit them in this paper.
For details, consult \cite[Sec. 3.3]{shao:bdc_nv}.
An earlier account for the special case of the geodesic foliation is presented in \cite{wang:cg}.

\begin{proposition}\label{thm.pnc_init}
The following limits hold for each $\omega \in \Sph^2$.
\begin{itemize}
\item We have the following comparisons between $f$ and the affine parameter $s$:
\begin{equation}\label{eq.pnc_init_fol_fct} \lim_{v \searrow 0} \valat{\frac{s}{f}}_{\paren{v, \omega}} = \vartheta_0 \text{,} \qquad \valat{\paren{\frac{1}{f} - \frac{\vartheta}{s}}}_{\paren{v, \omega}} = 0 \text{.} \end{equation}
Here, $\vartheta_0$ is the initial value of the null lapse $\vartheta$.

\item For each integer $k > 0$,
\begin{equation}\label{eq.pnc_init_ap} \lim_{v \searrow 0} \valat{v^{k-1} \abs{\nasla^k s}}_{\paren{v, \omega}} = 0 \text{,} \qquad \lim_{v \searrow 0} \valat{v^k \abs{\nasla^k \vartheta}}_{\paren{v, \omega}} = 0 \text{,} \end{equation}
where $\nasla^k$ denotes the $k$-th order horizontal covariant differential, i.e., the operator $\nasla$ applied successively $k$ times.
\footnote{Note that $\nasla$ here is defined with respect to the $f$-foliation.}

\item We have the following limits for $\chi$:
\begin{equation}\label{eq.pnc_init_chi} \lim_{v \searrow 0} \valat{\abs{\vartheta \paren{\trace \chi} - 2 f^{-1}}}_{\paren{v, \omega}} = 0 \text{,} \qquad \lim_{v \searrow 0} \valat{\abs{\hat{\chi}}}_{\paren{v, \omega}} = 0 \text{.} \end{equation}

\item We have the following limits for $\zeta$ and $\ul{\eta}$:
\begin{equation}\label{eq.pnc_init_zeta} \lim_{v \searrow 0} v \valat{\abs{\zeta}}_{\paren{v, \omega}} = \lim_{v \searrow 0} v \valat{\abs{\ul{\eta}}}_{\paren{v, \omega}} = 0 \text{.} \end{equation}

\item We also have the following limits for $\ul{\chi}$:
\begin{equation}\label{eq.pnc_init_chib} \lim_{v \searrow 0} v \valat{\paren{\trace \ul{\chi}}}_{\paren{v, \omega}} = -2 \vartheta_0^{-1} \text{,} \qquad \lim_{v \searrow 0} v \valat{\abs{\hat{\ul{\chi}}}}_{\paren{v, \omega}} = 0 \text{.} \end{equation}

\item Let $\phi \in C^\infty(\mc{N}^-(p))$, let $\phi_0 \in C^\infty(\Sph^2)$, and suppose
\[ \lim_{v \searrow 0} \valat{\phi}_{\paren{v, \omega^\prime}} = \valat{\phi_0}_{\paren{\omega^\prime}} \text{,} \qquad \omega^\prime \in \Sph^2 \text{.} \]
Then, the following integral limit holds:
\begin{equation}\label{eq.pnc_init_int} \lim_{v \searrow 0} v^{-2} \int_{\mc{S}_v} \phi = \vartheta_0^2 \int_{\Sph^2} \valat{\phi_0}_{\omega^\prime} d\omega^\prime \text{,} \end{equation}
where the integrals on the right hand sides are with respect to the standard Euclidean measure on $\Sph^2$.
\end{itemize}
\end{proposition}

In the previous proposition, the tensor norms $|\cdot|$ denote the natural Riemannian norms on the $\mc{S}_v$'s, that is, for any $A \in \Gamma \ul{T}^k \mc{N}^-(p)$, we define
\[ \abs{A}^2 = \lambda\paren{A, A} = \paren{\prod_{i=1}^k \lambda^{a_i b_i}} A_{a_1 \ldots a_k} A_{b_1 \ldots b_k} \text{.} \]

\begin{remark}
We can also see from Proposition \ref{thm.pnc_init} that the behaviors of the objects listed within the proposition all ``tend to their corresponding values in Minkowski space" at $p$.
For example, the expansion $\trace \chi$ in Minkowski space is precisely $2 s^{-1}$, which is the asymptotic behavior near $p$ demonstrated in \eqref{eq.pnc_init_chi}.
\end{remark}

\begin{remark}
We can also derive initial limits for derivatives of the Ricci coefficients.
For details, see \cite[Sec. 3.3]{shao:bdc_nv}.
\end{remark}

\begin{remark}
The initial values for the Ricci coefficients behave even better than in Proposition \ref{thm.pnc_init} in the case of the geodesic foliation; see \cite[Sec. 3.3]{shao:bdc_nv} or \cite{wang:cg}.
\end{remark}

\section{The Main Theorem Statement}

Before we can state the representation theorem, we must make a few additional definitions.
Again, we suppose $\mc{N}^-(p)$ is normalized by $\mf{t} \in T_p M$ and foliated by $f \in C^\infty(\mc{N}^-(p))$.
Recall that from $\mf{t}$ and $f$, we can define the null vector fields $L$ and $\ul{L}$, as well as local null frames $L$, $\ul{L}$, $e_1$, $e_2$ adapted to the $f$-foliation of $\mc{N}^-(p)$.

For each $m \in \{1, \ldots, n\}$, we define the extrinsic tensor fields
\[ \sss{B}{m} \in \Gamma \ol{T}^{\sss{r}{m}} \mc{N}^-(p) \]
such that they satisfy the coupled system of transport equations
\begin{equation}\label{eq.transport_pre} \ol{\nasla}_f \sss{B}{m}^I = -\frac{1}{2} \brak{\vartheta \paren{\trace \chi} - \frac{2}{f}} \sss{B}{m}^I + \frac{\vartheta}{2} \sum_{c = 1}^n \sss{P}{cm}_{4J}{}^I \sss{B}{c}^J \text{,} \qquad 1 \leq m \leq n \end{equation}
along the null generators of $\mc{N}^-(p)$, where $I$ and $J$ here denote collections of $\sss{r}{m}$ and $\sss{r}{c}$ extrinsic indices.
We also stipulate the initial conditions
\begin{equation}\label{eq.transport_init} \valat{\sss{B}{m}}_p = \sss{J}{m} \text{,} \qquad 1 \leq m \leq n \text{,} \end{equation}
where each $\sss{J}{m}$ is an aribtrary tensor of rank $\sss{r}{m}$ at $p$.
Note that the validity of the system \eqref{eq.transport_pre}, \eqref{eq.transport_init} follows from the initial limit \eqref{eq.pnc_init_chi} for $\trace \chi$.

We also define the extrinsic fields
\[ \sss{A}{m} = f^{-1} \sss{B}{m} \in \Gamma \ol{T}^{\sss{r}{m}} \mc{N}^-\paren{p} \text{,} \]
and we note that the $\sss{A}{m}$'s satisfy the transport equations
\begin{equation}\label{eq.transport} \ol{\nasla}_L \sss{A}{m}^I = -\frac{1}{2} \paren{\trace \chi} \sss{A}{m}^I + \frac{1}{2} \sum_{c = 1}^n \sss{P}{cm}_{4J}{}^I \sss{A}{c}^J \text{,} \qquad 1 \leq m \leq n \text{.} \end{equation}

The $\sss{A}{m}$'s are the analogues of the tensor field $A$ in \cite{kl_rod:ksp} and correspond to the factor $r^{-1} = |y - x|^{-1}$ in the explicit expression for $\phi_1$ in the model problem.
Moreover, if $n = 1$ and the $\sss{P}{mc}$'s vanish, then \eqref{eq.transport} reduces to the transport equation given in \cite{kl_rod:ksp}.
Note in particular that even if a specific $\sss{J}{m}$ vanishes, $\sss{A}{m}$ can still be nontrivial due to the coupling in \eqref{eq.transport}.

Lastly, we define the coefficients
\[ \sss{\nu}{cm} \in \Gamma \ol{T}^{\sss{r}{c} + \sss{r}{m}} \mc{N}^-\paren{p} \text{,} \qquad 1 \leq m, c \leq n \]
by the formula
\begin{align}
\label{eq.mass_asp_P} \sss{\nu}{cm}_J{}^I &= - \ol{\nasla}^a \sss{P}{cm}_{aJ}{}^I + \frac{1}{2} \ol{\nasla}_4 \sss{P}{cm}_{3J}{}^I + \zeta^a \sss{P}{cm}_{aJ}{}^I + \frac{1}{4} \paren{\trace \ul{\chi}} \sss{P}{cm}_{4J}{}^I \\
\notag &\qquad + \frac{1}{4} \paren{\trace \chi} \sss{P}{cm}_{3J}{}^I + \frac{1}{2} \sum_{d = 1}^n \sss{P}{cd}_{4J}{}^K \sss{P}{dm}_{3K}{}^I \text{.}
\end{align}
These tensor fields will be present in the error terms of the representation formula.

\subsection{The Main Theorem}

We are now ready to state the main theorem.

\begin{theorem}\label{thm.ksp_gen}
Assume the following:
\begin{itemize}
\item Let $n$ be a positive integer, let $\sss{r}{1}, \ldots, \sss{r}{n}$ be nonnegative integers, and suppose for each $1 \leq m, c \leq n$, we have defined tensor fields $\sss{\Phi}{m}$, $\sss{\Psi}{m}$, and $\sss{P}{mc}$ on $M$ of ranks $\sss{r}{m}$, $\sss{r}{m}$, and $1 + \sss{r}{m} + \sss{r}{c}$, respectively.

\item Suppose the $\sss{\Phi}{m}$'s, $\sss{\Psi}{m}$'s, and $\sss{P}{mc}$'s satisfy the system \eqref{eq.tensor_wave_system}.

\item Fix $p \in M$, and suppose a regular portion $\mc{N}^-(p)$ of the past null cone $N^-(p)$ is normalized and foliated by $\mf{t} \in T_p M$ and $f \in C^\infty(\mc{N}^-(p))$.

\item Let $v_0$ be a constant satisfying $0 < v_0 \leq \mf{i}(p)$.

\item For each $1 \leq m \leq n$, we define the extrinsic tensor fields
\[ \sss{B}{m} \in \Gamma \ol{T}^{\sss{r}{m}} \mc{N}^-\paren{p} \text{,} \qquad \sss{A}{m} = f^{-1} \sss{B}{m} \in \Gamma \ol{T}^{\sss{r}{m}} \mc{N}^-\paren{p} \text{,} \]
along with a tensor $\sss{J}{m}$ of rank $\sss{r}{m}$ at $p$, such that the systems of transport equations \eqref{eq.transport_pre}, \eqref{eq.transport_init}, \eqref{eq.transport} hold.
\end{itemize}
Then, we have the representation formula
\begin{equation}\label{eq.ksp_gen} 4 \pi \vartheta_0 \sum_{m = 1}^n \sss{J}{m}^I \valat{\sss{\Phi}{m}_I}_p = \mf{F}\paren{p; v_0} + \mf{E}^1\paren{p; v_0} + \mf{E}^2\paren{p; v_0} + \mf{I}\paren{p; v_0} \text{,} \end{equation}
where
\begin{itemize}
\item The ``fundamental solution term" $\mf{F}(p; v_0)$ is given by
\begin{equation}\label{eq.ksp_gen_F} \mf{F}\paren{p; v_0} = - \sum_{m = 1}^n \int_{\mc{N}^-\paren{p; v_0}} \sss{A}{m}^I \sss{\Psi}{m}_I \text{.} \end{equation}

\item The ``principal error terms" $\mf{E}^1(p; v_0)$ are given by
\begin{align}
\label{eq.ksp_gen_E1_app} \mf{E}^1\paren{p; v_0} &= - \sum_{m = 1}^n \int_{\mc{N}^-\paren{p; v_0}} \ol{\nasla}^a \sss{A}{m}^I \ol{\nasla}_a \sss{\Phi}{m}_I \\
\notag &\qquad + \sum_{m = 1}^n \int_{\mc{N}^-\paren{p; v_0}} \paren{\zeta^a - \ul{\eta}^a} \ol{\nasla}_a \sss{A}{m}^I \sss{\Phi}{m}_I \text{.}
\end{align}

\item The remaining ``error terms" $\mf{E}^2(p; v_0)$ are given by
\begin{align}
\label{eq.ksp_gen_E2} \mf{E}^2\paren{p; v_0} &= \sum_{m = 1}^n \int_{\mc{N}^-\paren{p; v_0}} \mu \cdot \sss{A}{m}^I \sss{\Phi}{m}_I \\
\notag &\qquad + \frac{1}{2} \sum_{m = 1}^n \int_{\mc{N}^-\paren{p; v_0}} \sss{A}{m}^I R_{43}\brak{\sss{\Phi}{m}}_I \\
\notag &\qquad - \sum_{m, c = 1}^n \int_{\mc{N}^-\paren{p; v_0}} \sss{P}{cm}_{aJ}{}^I \ol{\nasla}^a \sss{A}{c}^I \sss{\Phi}{m}_I \\
\notag &\qquad + \sum_{m, c = 1}^n \int_{\mc{N}^-\paren{p; v_0}} \sss{\nu}{cm}_J{}^I \cdot \sss{A}{c}^J \sss{\Phi}{m}_I \text{,}
\end{align}
where $R_{43}[\sss{\Phi}{m}]_I$ denotes the curvature quantity
\[ R_{43} \brak{\sss{\Phi}{m}}_I = D_{43} \sss{\Phi}{m}_I - D_{34} \sss{\Phi}{m}_I \text{,} \]
and where ``$D_{43}$" and ``$D_{34}$" denote second covariant derivatives.

\item The ``initial value terms" $\mf{I}(p; v_0)$ are given by
\begin{align}
\label{eq.ksp_gen_I} \mf{I}\paren{p; v_0} &= -\frac{1}{2} \sum_{m = 1}^n \int_{\mc{S}_{v_0}} \paren{\trace \ul{\chi}} \sss{A}{m}^I \sss{\Phi}{m}_I - \sum_{m = 1}^n \int_{\mc{S}_{v_0}} \sss{A}{m}^I D_3 \sss{\Phi}{m}_I \\
\notag &\qquad - \frac{1}{2} \sum_{m, c = 1}^n \int_{\mc{S}_{v_0}} \sss{P}{cm}_{3J}{}^I \sss{A}{c}^J \sss{\Phi}{m}_I \text{.}
\end{align}
\end{itemize}
Here, we have indexed with respect to arbitrary null frames $L, \ul{L}, e_1, e_2$ adapted to the $f$-foliation.
The capital letters $I, J$ refer to collections of extrinsic indices.

Furthermore, the error terms $\mf{E}^1(p; v_0)$ can be alternately expressed as
\begin{equation}\label{eq.ksp_gen_E1_ksp} \mf{E}^1\paren{p; v_0} = \sum_{m = 1}^n \int_{\mc{N}^-\paren{p; v_0}} \ol{\lasl} \sss{A}{m}^I \sss{\Phi}{m}_I + 2 \sum_{m = 1}^n \int_{\mc{N}^-\paren{p; v_0}} \zeta^a \cdot \ol{\nasla}_a \sss{A}{m}^I \sss{\Phi}{m}_I \text{.} \end{equation}
\end{theorem}

Note that $\mf{F}(p; v_0)$ corresponds to the explicit form of $\phi_1$ in the model problem, while $\mf{E}^1(p; v_0)$ and $\mf{E}^2(p; v_0)$ are error terms which vanish in the model problem.
Lastly, $\mf{I}(p; v_0)$ describes the initial value contributions, which correspond to the Kirchhoff formula for $\phi_2$ in the model problem.

\begin{remark}
Although the representation formula was stated in \eqref{eq.ksp_gen}-\eqref{eq.ksp_gen_E1_ksp} in index notation, this was done only as a matter of convenience.
It is easy to see that these expressions can in fact be described invariantly.
\end{remark}

\begin{remark}
In particular, we can use \eqref{eq.ksp_gen} to examine the value of any $\sss{\Phi}{m}|_p$ individually by setting $\sss{J}{c} = 0$ for all $c \neq m$.
\end{remark}

\subsection{Improvements over {\bf KR}}

Theorem \ref{thm.ksp_gen} contains a number of improvements over the parametrix {\bf KR}.
The most apparent is that Theorem \ref{thm.ksp_gen} handles the system \eqref{eq.tensor_wave_system} and not just the single equation \eqref{eq.tensor_wave}.
While \cite{kl_rod:ksp} only presented the case with vanishing initial data, here we present the general case for arbitrary initial data; the dependence on the initial data is given explicitly by the $\mf{I}(p; v_0)$-term \eqref{eq.ksp_gen_I}.

\begin{remark}
In fact, the extended parametrix presented in \cite[Thm. 4.1]{kl_rod:ksp} can handle the ``multiple wave equations" aspect of \eqref{eq.tensor_wave_system}, but it does not cover the issues of general first-order terms.
For details, see section 5.
\end{remark}

Additionally, Theorem \ref{thm.ksp_gen} weakens the assumptions required for the representation formula to be valid.
The main assumptions needed for {\bf KR} are the local spacetime conditions {\bf A1} and {\bf A2} given in \cite[Sec. 2.1]{kl_rod:ksp}.
First, since we make no reference to the optical function and remain entirely on $\mc{N}^-(p)$ in our derivation, we only need that the null cone regularity stated in {\bf A2} hold for $\mc{N}^-(p)$, rather than for $\mc{N}^-(q)$ for every $q$ in a neighborhood of $p$, as needed in \cite{kl_rod:ksp}.
This assumption is reflected in Theorem \ref{thm.ksp_gen} by working only within the segment $\mc{N}^-(p; v_0)$, where $v_0 \leq \mf{i}(p)$.
Moreover, since we work exclusively on $\mc{N}^-(p)$, the local hyperbolicity condition {\bf A1} is entirely superfluous here.

\subsection{Exclusive Dependence on the Null Cone}

We also previously noted that the representation formula depends only on quantities defined on $\mc{N}^-(p)$.
Indeed, by inspecting the individual portions \eqref{eq.ksp_gen_F}-\eqref{eq.ksp_gen_E2}, we can immediately see that all the objects within the integrands can be expressed as horizontal, extrinsic, or mixed tensor fields on $\mc{N}^-(p)$.
The representation formula remains unaffected by potential extensions of these quantities off $\mc{N}^-(p)$.
\footnote{The only exceptions to these rules are the factors $D_3 \sss{\Phi}{m}$ in the ``initial data" terms \eqref{eq.ksp_gen_I}, which depend on transversal derivatives of the $\sss{\Phi}{m}$'s.}

We remark that this property of the parametrix does not conflict with the fact that wave equations on curved spacetimes do not satisfy the strong Huygens principle.
This is demonstrated in Theorem \ref{thm.ksp_gen} by the recursive error terms \eqref{eq.ksp_gen_E1_app}, \eqref{eq.ksp_gen_E2}, \eqref{eq.ksp_gen_E1_ksp}, which contain the unknowns $\sss{\Phi}{1}, \ldots, \sss{\Phi}{n}$ themselves.

\subsection{Extensions Beyond Cut Locus Terminal Points}

As stated in Theorem \ref{thm.ksp_gen}, the parametrix is valid only when the past null exponential $\ul{\exp}_p$ remains a global diffeomorphism.
We can, however, trivially extend Theorem \ref{thm.ksp_gen} beyond cut locus terminal points, that is, terminal points resulting from the intersection of two distinct null generators, by lifting the representation formula to the tangent space $T_p M$ via $\ul{\exp}_p$.
By pulling back over $\ul{\exp}_p$, we can consider all the tensorial objects referenced in Theorem \ref{thm.ksp_gen} as fields on the past null cone $\mf{N}$ of $T_p M$.
At the level of the tangent space, cut locus terminal points no longer exist, so we can apply Theorem \ref{thm.ksp_gen} as usual to obtain a representation formula on $\mf{N}$.

Using $\ul{\exp}_p$, we can express the above representation formula on $\mf{N}$ back in terms of the past null cone of $p$, beyond any cut locus terminal points.
Due to intersecting null geodesics, some objects in the integrands, such as the Ricci coefficients, will be multi-valued at cut locus points.
However, in this extended formulation of Theorem \ref{thm.ksp_gen}, we longer have the standard local energy estimates which were crucial to the general relativity applications \cite{kl_rod:bdc, shao:bdc_nv}.

On the other hand, conjugate terminal points on past null cones pose a much more serious problem.
Such points by definition are preserved by the lifting via $\ul{\exp}_p$ to the tangent space $T_p M$.
Moreover, conjugate points are accompanied by the degeneration of the expansion $\trace \chi$ and hence the $\sss{A}{m}$'s.
Consequently, we have no natural extension of Theorem \ref{thm.ksp_gen} which survives beyond conjugate points.

\subsection{Applications to General Relativity}

Here, we briefly describe how Theorem \ref{thm.ksp_gen} is applied to the breakdown problem for the Einstein-Maxwell equations, i.e., to Theorem \ref{thm.bdc_maxwell}.
First, in the vacuum analogue, one crucial estimate is the uniform bound for the spacetime Riemann curvature $R$, which satisfies a wave equation
\begin{equation}\label{eq.wave_curv_vacuum} \Box_g R \cong R \cdot R \text{.} \end{equation}
The parametrix {\bf KR} was applied to \eqref{eq.wave_curv_vacuum} in order to control $R$ at a point $p$; the principal term to be estimated as a result is an integral over $\mc{N}^-(p)$ of the quadratic nonlinearity $R \cdot R$.
\footnote{The analogue $A$ of the transport equation in {\bf KR} also appears here.}
This was handled using a similar strategy as in the classical work \cite{ea_mo:g_ymh} of Eardley and Moncrief, in that at least one of the $R$'s in the integrand could be controlled by the flux density of $R$ on $\mc{N}^-(p)$.

The Einstein-Maxwell case is similar, except that we now have a system of two wave equations for both $R$ and the Maxwell $2$-form $F$ on $M$:
\begin{align}
\label{eq.wave_curv_em} \Box_g R &\cong F \cdot D^2 F + \paren{R + D F} \cdot \paren{R + D F} + \text{l.o.} \text{,} \\
\notag \Box_g D F &\cong F \cdot D R + \paren{R + D F} \cdot \paren{R + D F} + \text{l.o.} \text{.}
\end{align}
We can apply Theorem \ref{thm.ksp_gen} to this system, and we can handle the quadratic nonlinearities in the same manner as in the vacuum analogue.
The first-order terms $F \cdot D^2 F$ and $F \cdot D R$ present a new challenge, however, as the methods in \cite{kl_rod:bdc} using {\bf KR} fail here.
However, these terms can be handled directly using Theorem \ref{thm.ksp_gen}, which absorbs the effect of these terms into the transport equation.

More specifically, in the setting of Theorem \ref{thm.ksp_gen}, we take $\sss{\Phi}{1} = R$ and $\sss{\Phi}{2} = D F$.
From \eqref{eq.wave_curv_em}, we see that the first-order coefficients $\sss{P}{11}$ and $\sss{P}{22}$ vanish, while the coefficients $\sss{P}{12}$ and $\sss{P}{21}$ are schematically of the form $F$.
Due to the extra control present for $F$ stipulated in Theorem \ref{thm.bdc_maxwell}, the $\sss{A}{m}$'s, which depend on the $\sss{P}{cm}$'s, can still be adequately controlled.
For details, see \cite{shao:bdc_nv}.

\section{Proof of the Main Theorem}

In this section, we assume the hypotheses of Theorem \ref{thm.ksp_gen}, and we derive in detail the parametrix of Theorem \ref{thm.ksp_gen}.
Let $0 < \epsilon < v_0$, and define for convenience
\[ \mc{N}_\epsilon = \mc{N}^-\paren{p; \epsilon, v_0} = \brac{q \in \mc{N}^-\paren{p} \mid \epsilon < f\paren{q} < v_0} \text{.} \]

\subsection{General Integration Formulas}

The most fundamental steps in the derivation of Theorem \ref{thm.ksp_gen} involve integrations by parts.
Here, we state the general lemmas representing these steps.
These lemmas deal with two separate cases: integrations by parts involving $\ol{\nasla}$, and integrations by parts involving $\ol{\nasla}_L$.

We first examine the horizontal case.
This is a simple consequence of the compatibility of the mixed covariant differential $\ol{\nasla}$ with both the spacetime metric $g$ and the horizontal metrics $\lambda$.

\begin{lemma}\label{thm.int_parts_pnc_hor}
For any integer $r \geq 0$, $S \in \Gamma \ol{T}^r \mc{N}^-(p)$, and $T \in \Gamma \ul{T}^1 \ol{T}^r \mc{N}^-(p)$,
\[ \int_{\mc{N}_\epsilon} S^I \cdot \ol{\nasla}^a T_{aI} = -\int_{\mc{N}_\epsilon} \ol{\nasla}^a S^I \cdot T_{aI} - \int_{\mc{N}_\epsilon} \nasla^a \paren{\log \vartheta} \cdot S^I T_{aI} \text{,} \]
where $I$ denotes a collection of $r$ arbitrary extrinsic indices.
\end{lemma}

\begin{proof}
Applying Proposition \ref{thm.pnc_int} yields
\begin{align}
\label{eql.int_parts_pnc_hor_1} \int_{\mc{N}_\epsilon} S^I \cdot \ol{\nasla}^a T_{aI} &= \int_\epsilon^{v_0} \paren{\int_{\mc{S}_v} \vartheta \cdot S^I \ol{\nasla}^a T_{aI}} dv \\
\notag &= \int_\epsilon^{v_0} \brak{\int_{\mc{S}_v} \ol{\nasla}^a \paren{\vartheta \cdot S^I T_{aI}}} dv - \int_\epsilon^{v_0} \paren{\int_{\mc{S}_v} \vartheta \ol{\nasla}^a S^I \cdot T_{aI}} dv \\
\notag &\qquad - \int_\epsilon^{v_0} \paren{\int_{\mc{S}_v} \vartheta \nasla^a \paren{\log \vartheta} \cdot S^I T_{aI}} dv \text{.}
\end{align}
In the last step, we used that both $\ol{\nasla} g$ and $\ol{\nasla} \gamma$ vanish.
By Proposition \ref{thm.pnc_int} again, the last two terms on the right-hand side of \eqref{eql.int_parts_pnc_hor_1} are precisely the desired terms.
Defining the horizontal $1$-form $\omega$ by $\omega(X) = \vartheta \cdot S^I T_{aI} X^a$, then by the divergence theorem, the first term on the right hand side of \eqref{eql.int_parts_pnc_hor_1} becomes
\[ \int_\epsilon^{v_0} \paren{\int_{\mc{S}_v} \nasla^a \omega_a} dv = 0 \text{.} \qedhere \]
\end{proof}

As a special case of Lemma \ref{thm.int_parts_pnc_hor}, we have the following:

\begin{corollary}\label{thm.int_parts_pnc_lasl}
For any integer $r \geq 0$ and $S, T \in \Gamma \ol{\mc{T}}^r \mc{N}^-(p)$,
\[ \int_{\mc{N}_\epsilon} S^I \cdot \ol{\lasl} T_I = -\int_{\mc{N}_\epsilon} \ol{\nasla}^a S^I \cdot \ol{\nasla}_a T_I - \int_{\mc{N}_\epsilon} \nasla^a \paren{\log \vartheta} \cdot S^I \ol{\nasla}_a T_I \text{,} \]
where $I$ denotes a collection of $r$ arbitrary extrinsic indices.
\end{corollary}

Next, we give the result for the null direction case:

\begin{lemma}\label{thm.int_parts_pnc_null}
For any integer $r \geq 0$ and $S, T \in \Gamma \ol{\mc{T}}^r \mc{N}^-(p)$,
\[ \int_{\mc{N}_\epsilon} S^I \cdot \ol{\nasla}_L T_I = -\int_{\mc{N}_\epsilon} \ol{\nasla}_L S^I \cdot T_I - \int_{\mc{N}_\epsilon} \paren{\trace \chi} \cdot S^I T_I - \int_{\mc{S}_\epsilon} S^I T_I + \int_{\mc{S}_{v_0}} S^I T_I \text{,} \]
where $I$ denotes a collection of $r$ arbitrary extrinsic indices.
\end{lemma}

\begin{proof}
\footnote{Thanks to Qian Wang and her comments for greatly simplifying the proof.}
First, we have
\[ \int_{\mc{N}_\epsilon} S^I \cdot \ol{\nasla}_L T_I = \int_{\mc{N}_\epsilon} \ol{\nasla}_L \paren{S^I T_I} - \int_{\mc{N}_{\epsilon}} \ol{\nasla}_L S^I \cdot T_I \text{.} \]
By Propositions \ref{thm.pnc_int} and \ref{thm.evolution_pnc_int}, the first term on the right-hand side can be handled by
\begin{align*}
\int_{\mc{N}_\epsilon} \ol{\nasla}_L \paren{S^I T_I} &= \int_\epsilon^{v_0} \brak{\int_{\mc{S}_v} \ol{\nasla}_f \paren{S^I T_I}} dv \\
&= \int_\epsilon^{v_0} \valat{\frac{d}{dv^\prime} \paren{\int_{\mc{S}_{v^\prime}} S^I T_I}}_{v^\prime = v} dv - \int_\epsilon^{v_0} \brak{\int_{\mc{S}_v} \vartheta \paren{\trace \chi} S^I T_I} dv \text{,}
\end{align*}
where we recalled that $\ol{\nasla}_f$ denotes the covariant derivative $\vartheta \ol{\nasla}_L$.
The first term on the right-hand side is precisely the desired ``boundary terms".
Finally, an application of Proposition \ref{thm.pnc_int} to the second term completes the proof.
\end{proof}

\subsection{Expansion of the Differential Operator}

For any tensor field $T$ of rank $l$ on $M$, we have on $\mc{N}^-(p)$ the identity
\begin{align*}
\square_g T_I &= -\frac{1}{2} D_{43} T_I - \frac{1}{2} D_{34} T_I + \lambda^{ab} D_{ab} T_I \\
&= -D_{43} T_I + \frac{1}{2} R_{43} \brak{T}_I + \lambda^{ab} D_{ab} T_I \\
&= -\ol{\nasla}_4 \paren{D_3 T}_I + 2 \ul{\eta}^a \ol{\nasla}_a T_I + \frac{1}{2} R_{43} \brak{T}_I + \lambda^{ab} D_{ab} T_I \text{,}
\end{align*}
where $\ol{\nasla}_4$ denotes the operator $\ol{\nasla}_L$, where $R_{43}[T]_I$ denotes the curvature quantity $D_{43} T_I - D_{34} T_I$, and where in the last step, we applied the identity $D_L \ul{L} = 2 \ul{\eta}^a e_a$.

By the definitions of the connections $D$, $\ol{D}$, and $\ol{\nasla}$, we have the identities
\begin{align*}
D_{ab} T_I &= D_a \paren{D_b T}_I - D_{\ol{D}_a e_b} T_I \text{,} \\
\ol{\nasla}_{ab} T_I &= \ol{\nasla}_a \paren{\ol{\nasla}_b T}_I - \ol{\nasla}_{\ol{\nasla}_a e_b} T_I = D_a \paren{D_b T}_I - D_{\nasla_a e_b} T_I \text{,}
\end{align*}
so by \eqref{eq.D_nf}, we have the relation
\[ D_{ab} T_I = \ol{\nasla}_{ab} T_I - D_{\ol{D}_a e_b - \nasla_a e_b} T_I = \ol{\nasla}_{ab} T_I - \frac{1}{2} \ul{\chi}_{ab} \ol{\nasla}_4 T_I - \frac{1}{2} \chi_{ab} D_3 T_I \text{,} \]
As a result, we have the decomposition
\begin{align}
\label{eq.square_expand} \square T_I &= \ol{\lasl} T_I - \ol{\nasla}_4 \paren{D_3 T}_I + 2 \ul{\eta}^a \ol{\nasla}_a T_I - \frac{1}{2} \paren{\trace \ul{\chi}} \ol{\nasla}_4 T_I \\
\notag &\qquad - \frac{1}{2} \paren{\trace \chi} D_3 T_I + \frac{1}{2} R_{43} \brak{T}_I \text{.}
\end{align}

As a result, the quantity
\begin{align*}
\mf{F}^{\mc{L}}_\epsilon &= \sum_{m = 1}^n \int_{\mc{N}_\epsilon} \sss{A}{m}{}^I \sss{\Psi}{m}_I \\
&= \sum_{m = 1}^n \int_{\mc{N}_\epsilon} \sss{A}{m}{}^I \square_g \sss{\Phi}{m}_I + \sum_{m, c = 1}^n \int_{\mc{N}_\epsilon} \sss{A}{m}{}^I \sss{P}{mc}_{\mu I}{}^J D^\mu \sss{\Phi}{c}_J \text{,}
\end{align*}
where we have used \eqref{eq.tensor_wave_system}, can be expanded by \eqref{eq.square_expand} as 
\begin{align*}
\mf{F}^{\mc{L}}_\epsilon &= \sum_{m = 1}^n \int_{\mc{N}_\epsilon} \sss{A}{m}^I \brak{\ol{\lasl} \sss{\Phi}{m}_I - \ol{\nasla}_4 \paren{D_3 \sss{\Phi}{m}}_I + 2 \ul{\eta}^a \ol{\nasla}_a \sss{\Phi}{m}_I + \frac{1}{2} R_{43} \brak{\sss{\Phi}{m}}_I} \\
&\qquad + \sum_{m = 1}^n \int_{\mc{N}_\epsilon} \sss{A}{m}^I \brak{-\frac{1}{2} \paren{\trace \ul{\chi}} \ol{\nasla}_4 \sss{\Phi}{m}_I - \frac{1}{2} \paren{\trace \chi} D_3 \sss{\Phi}{m}_I} \\
&\qquad + \sum_{m, c = 1}^n \int_{\mc{N}_\epsilon} \sss{A}{c}{}^J \sss{P}{cm}_{\mu J}{}^I D^\mu \sss{\Phi}{m}_I \\
&= \mf{A}^1_\epsilon + \mf{A}^2_\epsilon + \mf{A}^3_\epsilon + \mf{F}^R_\epsilon + \mf{A}^4_\epsilon + \mf{A}^5_\epsilon + \mf{A}^6_\epsilon \text{.}
\end{align*}
Moreover, we can expand $\mf{A}^6_\epsilon$ as
\begin{align*}
\mf{A}^6_\epsilon &= \sum_{m, c = 1}^n \int_{\mc{N}_\epsilon} \sss{A}{c}{}^J \sss{P}{cm}_{aJ}{}^I \ol{\nasla}^a \sss{\Phi}{m}_I \\
&\qquad + \sum_{m, c = 1}^n \int_{\mc{N}_\epsilon} \sss{A}{c}{}^J \paren{-\frac{1}{2} \sss{P}{cm}_{3J}{}^I \ol{\nasla}_4 \sss{\Phi}{m}_I - \frac{1}{2} \sss{P}{cm}_{4J}{}^I D_3 \sss{\Phi}{m}_I} \\
&= \mf{C}^1_\epsilon + \mf{C}^2_\epsilon + \mf{C}^3_\epsilon \text{.}
\end{align*}

\subsection{Integrations by Parts}

The main goal in principle is as follows.
As of now, all the existing covariant derivatives thus far are acting on the $\sss{\Phi}{m}$'s.
We wish to move all covariant derivatives in directions tangent to $\mc{N}^-(p)$ (that is, $L$, $e_1$, and $e_2$) away from the $\sss{\Phi}{m}$'s.
This will be accomplished through multiple applications of Lemmas \ref{thm.int_parts_pnc_hor} and \ref{thm.int_parts_pnc_null}.
We will then see that the transport equation \eqref{eq.transport} will eliminate the worst remaining terms: those involving $\ul{L}$-derivatives of the $\sss{\Phi}{m}$'s.

If we apply Corollary \ref{thm.int_parts_pnc_lasl} to $\mf{A}^1_\epsilon$, we obtain
\[ \mf{A}^1_\epsilon = \sum_{m = 1}^n \int_{\mc{N}_\epsilon} \brak{- \ol{\nasla}^a \sss{A}{m}^I \ol{\nasla}_a \sss{\Phi}{m}_I - \ol{\nasla}^a \paren{\log \vartheta} \sss{A}{m}^I \ol{\nasla}_a \sss{\Phi}{m}_I} \text{.} \]
Combining this with \eqref{eq.null_torsion}, then
\begin{align*}
\mf{A}^1_\epsilon + \mf{A}^3_\epsilon &= \sum_{m = 1}^n \brak{- \int_{\mc{N}_\epsilon} \ol{\nasla}^a \sss{A}{m}^I \ol{\nasla}_a \sss{\Phi}{m}_I + \int_{\mc{N}_\epsilon} \paren{\ul{\eta}^a - \zeta^a} \sss{A}{m}^I \ol{\nasla}_a \sss{\Phi}{m}_I} \\
&= \mf{F}^\nabla_\epsilon + \mf{B}^1_\epsilon \text{.}
\end{align*}
Applying Lemma \ref{thm.int_parts_pnc_hor} to $\mf{B}^1_\epsilon$ yields
\begin{align*}
\mf{B}^1_\epsilon &= \sum_{m = 1}^n \int_{\mc{N}_\epsilon} \sss{A}{m}^I \brak{\nasla_a\paren{\log \vartheta} \cdot \paren{\zeta^a - \ul{\eta}^a} \sss{\Phi}{m}_I + \paren{\nasla^a \zeta_a - \nasla^a \ul{\eta}_a} \sss{\Phi}{m}_I} \\
&\qquad + \sum_{m = 1}^n \int_{\mc{N}_\epsilon} \paren{\zeta^a - \ul{\eta}^a} \ol{\nasla}_a \sss{A}{m}^I \sss{\Phi}{m}_I \\
&= \mf{B}^2_\epsilon + \mf{B}^3_\epsilon + \mf{F}^\zeta_\epsilon \text{.}
\end{align*}

Likewise, we apply Lemma \ref{thm.int_parts_pnc_null} to $\mf{A}^2_\epsilon$ to obtain
\begin{align*}
\mf{A}^2_\epsilon &= \sum_{m = 1}^n \brak{\int_{\mc{N}_\epsilon} \ol{\nasla}_4 \sss{A}{m}^I D_3 \sss{\Phi}{m}_I + \int_{\mc{N}_\epsilon} \paren{\trace \chi} \sss{A}{m}^I D_3 \sss{\Phi}{m}_I} \\
&\qquad + \sum_{m = 1}^n \paren{\int_{\mc{S}_\epsilon} \sss{A}{m}^I D_3 \sss{\Phi}{m}_I - \int_{\mc{S}_{v_0}} \sss{A}{m}^I D_3 \sss{\Phi}{m}_I} \text{.}
\end{align*}
Taking note of the transport equation \eqref{eq.transport}, then
\[ \mf{A}^2_\epsilon + \mf{A}^5_\epsilon + \mf{C}^3_\epsilon = \sum_{m = 1}^n \paren{\int_{\mc{S}_\epsilon} \sss{A}{m}^I D_3 \sss{\Phi}{m}_I - \int_{\mc{S}_{v_0}} \sss{A}{m}^I D_3 \sss{\Phi}{m}_I} = \mf{L}^1_\epsilon + \mf{I}^1 \text{.} \]

\begin{remark}
The above computation also provides the justification for the transport equation \eqref{eq.transport}.
Indeed, this is exactly what is needed to eliminate the terms involving integrals over $\mc{N}_\epsilon$ of the $D_3 \sss{\Phi}{m}$'s.
Recall from standard local energy estimates that $\ul{L}$-derivatives on the null cone are generally the worst behaved.
\end{remark}

\begin{remark}
Recall also that \cite{kl_rod:ksp} justified its transport equation as the condition needed to eliminate terms involving the distribution $\delta^\prime(u)$, where $u$ is the ``optical function" used throughout \cite{kl_rod:ksp}.
Note, however, that $\delta^\prime(u)$ is precisely $-\frac{1}{2} \cdot D_3 \delta(u)$, which corresponds to an integral over the null cone of some terms, with the principal term being the $\ul{L}$-derivative of the test function.
\end{remark}

We also apply Lemma \ref{thm.int_parts_pnc_null} to $\mf{A}^4_\epsilon$:
\begin{align*}
\mf{A}^4_\epsilon &= \sum_{m = 1}^n \int_{\mc{N}_\epsilon} \brak{\frac{1}{2} \nasla_4 \paren{\trace \ul{\chi}} \sss{A}{m}^I \sss{\Phi}{m}_I + \frac{1}{2} \paren{\trace \ul{\chi}} \ol{\nasla}_4 \sss{A}{m}^I \sss{\Phi}{m}_I} \\
&\qquad + \sum_{m = 1}^n \int_{\mc{N}_\epsilon} \frac{1}{2} \paren{\trace \chi} \paren{\trace \ul{\chi}} \sss{A}{m}^I \sss{\Phi}{m}_I \\
&\qquad + \sum_{m = 1}^n \brak{\frac{1}{2} \int_{\mc{S}_\epsilon} \paren{\trace \ul{\chi}} \sss{A}{m}^I \sss{\Phi}{m}_I - \frac{1}{2} \int_{\mc{S}_{v_0}} \paren{\trace \ul{\chi}} \sss{A}{m}^I \sss{\Phi}{m}_I} \\
&= \mf{B}^4_\epsilon + \mf{B}^5_\epsilon + \mf{B}^6_\epsilon + \mf{L}^0_\epsilon + \mf{I}^0 \text{.}
\end{align*}
We can further expand $\mf{B}^4_\epsilon$ using \eqref{eq.str_ev_chib_tr}:
\begin{align*}
\mf{B}^4_\epsilon + \mf{B}^6_\epsilon &= \sum_{m = 1}^n \int_{\mc{N}_\epsilon} \brak{\nasla^a \ul{\eta}_a + \abs{\ul{\eta}}^2 + \frac{1}{4} \paren{\trace \chi} \paren{\trace \ul{\chi}} - \frac{1}{2} \hat{\chi}^{ab} \hat{\ul{\chi}}_{ab}} \sss{A}{m}^I \sss{\Phi}{m}_I \\
&\qquad + \sum_{m = 1}^n \int_{\mc{N}_\epsilon} \paren{\frac{1}{4} R_{4343} - \frac{1}{2} R_{43}} \sss{A}{m}^I \sss{\Phi}{m}_I \\
&= \mf{B}^7_\epsilon + \mf{B}^8_\epsilon + \mf{B}^9_\epsilon + \mf{M}^1_\epsilon + \mf{M}^2_\epsilon + \mf{M}^3_\epsilon \text{.}
\end{align*}
By the transport equation \eqref{eq.transport}, then
\[ \mf{B}^5_\epsilon + \mf{B}^9_\epsilon = \frac{1}{4} \sum_{\mf{m}, \mf{c} = 1}^\mf{n} \int_{\mc{N}_\epsilon} \paren{\trace \ul{\chi}} \sss{P}{cm}_{4J}{}^I \sss{A}{c}^J \sss{\Phi}{m}_I = \mf{P}^1_\epsilon \text{.} \]

If we apply Lemma \ref{thm.int_parts_pnc_hor} to $\mf{C}^1_\epsilon$, we obtain
\begin{align*}
\mf{C}^1_\epsilon &= \sum_{m, c = 1}^n \int_{\mc{N}_\epsilon} \sss{A}{c}^J \brak{-\nasla^a \paren{\log \vartheta} \cdot \sss{P}{cm}_{aJ}{}^I \sss{\Phi}{m}_I - \ol{\nasla}^a \sss{P}{cm}_{aJ}{}^I \sss{\Phi}{m}_I} \\
&\qquad - \sum_{m, c = 1}^n \int_{\mc{N}_\epsilon} \sss{P}{cm}_{aJ}{}^I \ol{\nasla}^a \sss{A}{c}^J \sss{\Phi}{m}_I \\
&= \mf{C}^4_\epsilon + \mf{P}^2_\epsilon + \mf{F}^{\nabla P}_\epsilon \text{.}
\end{align*}
Moreover, applying Lemma \ref{thm.int_parts_pnc_null} to $\mf{C}^2_\epsilon$, we have
\begin{align*}
\mf{C}^2_\epsilon &= - \frac{1}{2} \sum_{m, c = 1}^n \int_{\mc{N}_\epsilon} \ul{L}^\alpha \sss{A}{c}^J \sss{P}{cm}_{\alpha J}{}^I \ol{\nasla}_4 \sss{\Phi}{m}_I \\
&= \sum_{m, c = 1}^n \int_{\mc{N}_\epsilon} \brak{\frac{1}{2} \paren{\trace \chi} \sss{P}{cm}_{3J}{}^I \sss{A}{c}^J \sss{\Phi}{m}_I + \ul{\eta}^a \sss{P}{cm}_{aJ}{}^I \sss{A}{c}^J \sss{\Phi}{m}_I} \\
&\qquad + \sum_{m, c = 1}^n \int_{\mc{N}_\epsilon} \paren{\frac{1}{2} \ol{\nasla}_4 \sss{P}{cm}_{3J}{}^I \sss{A}{c}^J \sss{\Phi}{m}_I + \frac{1}{2} \sss{P}{cm}_{3J}{}^I \ol{\nasla}_4 \sss{A}{c}^J \sss{\Phi}{m}_I} \\
&\qquad + \sum_{m, c = 1}^n \paren{\frac{1}{2} \int_{\mc{S}_\epsilon} \sss{P}{cm}_{3J}{}^I \sss{A}{c}^J \sss{\Phi}{m}_I - \frac{1}{2} \int_{\mc{S}_{v_0}} \sss{P}{cm}_{3J}{}^I \sss{A}{c}^J \sss{\Phi}{m}_I} \\
&= \mf{C}^5_\epsilon + \mf{C}^6_\epsilon + \mf{P}^3_\epsilon + \mf{C}^7_\epsilon + \mf{L}^2_\epsilon + \mf{I}^2 \text{.}
\end{align*}
Applying the transport equation \eqref{eq.transport}, we have
\begin{align*}
\mf{C}^5_\epsilon + \mf{C}^7_\epsilon &= \sum_{m, c = 1}^n \frac{1}{4} \int_{\mc{N}_\epsilon} \paren{\trace \chi} \sss{P}{cm}_{3J}{}^I \sss{A}{c}^J \sss{\Phi}{m}_I \\
&\qquad + \sum_{m, c, d = 1}^n \frac{1}{4} \int_{\mc{N}_\epsilon} \sss{P}{cd}_{4J}{}^K \sss{P}{dm}_{3K}{}^I \sss{A}{c}^J \sss{\Phi}{m}_I \\
&= \mf{P}^4_\epsilon + \mf{P}^5_\epsilon \text{.}
\end{align*}

As of now, we have the rather unsightly expansion
\begin{align}
\label{eql.ksp_gen_ibp} \mf{F}^{\mc{L}}_\epsilon &= \mf{F}^\nabla_\epsilon + \mf{F}^{\nabla P}_\epsilon + \mf{F}^\zeta_\epsilon + \mf{F}^R_\epsilon + \mf{B}^2_\epsilon + \mf{B}^3_\epsilon + \mf{B}^7_\epsilon + \mf{B}^8_\epsilon + \mf{C}^4_\epsilon + \mf{C}^6_\epsilon + \mf{M}^1_\epsilon + \mf{M}^2_\epsilon \\
\notag &\qquad + \mf{M}^3_\epsilon + \mf{P}^1_\epsilon + \mf{P}^2_\epsilon + \mf{P}^3_\epsilon + \mf{P}^4_\epsilon + \mf{P}^5_\epsilon + \mf{L}^0_\epsilon + \mf{L}^1_\epsilon + \mf{L}^2_\epsilon + \mf{I}^0 + \mf{I}^1 + \mf{I}^2 \text{.}
\end{align}

\subsection{The Error Terms}

The next step is to aggregate the terms in \eqref{eql.ksp_gen_ibp} into error terms corresponding to those comprising \eqref{eq.ksp_gen}.
First, recalling \eqref{eq.null_torsion}, we see that
\begin{align*}
\mf{B}^3_\epsilon + \mf{B}^7_\epsilon &= \sum_{m = 1}^n \int_{\mc{N}_\epsilon} \paren{\nasla^a \zeta_a} \sss{A}{m}^I \sss{\Phi}{m}_I = \mf{M}^4_\epsilon \text{,} \\
\mf{B}^2_\epsilon + \mf{B}^8_\epsilon &= \sum_{m = 1}^n \int_{\mc{N}_\epsilon} \abs{\zeta}^2 \sss{A}{m}^I \sss{\Phi}{m}_I = \mf{M}^5_\epsilon \text{.}
\end{align*}
Recalling the mass aspect function $\mu$ defined in \eqref{eq.maf}, we have
\[ \mf{M}^1_\epsilon + \mf{M}^2_\epsilon + \mf{M}^3_\epsilon + \mf{M}^4_\epsilon + \mf{M}^5_\epsilon = \sum_{m = 1}^n \int_{\mc{N}_\epsilon} \mu \sss{A}{m}^I \sss{\Phi}{m}_I = \mf{F}^\mu_\epsilon \text{.} \]
Next, by \eqref{eq.null_torsion},
\[ \mf{C}^4_\epsilon + \mf{C}^6_\epsilon = - \sum_{m, c = 1}^n \int_{\mc{N}_\epsilon} \zeta^a \sss{P}{cm}_{aJ}{}^I \sss{A}{c}^J \sss{\Phi}{m}_I \text{,} \]
hence we obtain
\[ \mf{C}^4_\epsilon + \mf{C}^6_\epsilon + \mf{P}^1_\epsilon + \mf{P}^2_\epsilon + \mf{P}^3_\epsilon + \mf{P}^4_\epsilon + \mf{P}^5_\epsilon = \sum_{m, c = 1}^n \int_{\mc{N}_\epsilon} \sss{\nu}{cm}_J{}^I \sss{A}{c}^J \sss{\Phi}{m}_I = \mf{F}^\nu_\epsilon \text{.} \]

Finally, we obtain the more manageable equation
\begin{equation}\label{eql.ksp_gen_aux} \mf{F}^{\mc{L}}_\epsilon = \mf{F}^\nabla_\epsilon + \mf{F}^\zeta_\epsilon + \mf{F}^{\nabla P}_\epsilon + \mf{F}^\mu_\epsilon + \mf{F}^\nu_\epsilon + \mf{F}^R_\epsilon + \mf{L}^0_\epsilon + \mf{L}^1_\epsilon + \mf{L}^2_\epsilon + \mf{I}^0 + \mf{I}^1 + \mf{I}^2 \text{.} \end{equation}
Note that:
\begin{itemize}
\item The ``$\mf{F}$"-terms are integrals over the regular past null cone segment $\mc{N}_\epsilon$.

\item The ``$\mf{L}$"-terms are the vertex limit terms, expressed as integrals over $\mc{S}_\epsilon$.

\item The ``$\mf{I}$"-terms are the initial data terms, expressed as integrals over $\mc{S}_{v_0}$.
\end{itemize}
The limits as $\epsilon \searrow 0$ of the ``$\mf{F}$"-terms in \eqref{eql.ksp_gen_aux} are clear - these involve simply replacing the integral over $\mc{N}_\epsilon$ by the same integral over $\mc{N}^-(p; v_0)$.
These account for the terms $\mf{F}(p; v_0)$, $\mf{E}^1(p; v_0)$, and $\mf{E}^2(p; v_0)$ in Theorem \ref{thm.ksp_gen}.
Moreover, the ``$\mf{I}$"-terms in \eqref{eql.ksp_gen_aux} correspond exactly to $\mf{I}(p; v_0)$ in Theorem \ref{thm.ksp_gen}.

As a result, the limit as $\epsilon \searrow 0$ of \eqref{eql.ksp_gen_aux} evaluates to
\begin{equation}\label{eql.ksp_gen_vl} - \lim_{\epsilon \searrow 0} \paren{\mf{L}^0_\epsilon + \mf{L}^1_\epsilon + \mf{L}^2_\epsilon} = \mf{F}\paren{p; v_0} + \mf{E}^1\paren{p; v_0} + \mf{E}^2\paren{p; v_0} + \mf{I}\paren{p; v_0} \text{.} \end{equation}
Now, it remains only to compute the limits on the left-hand side of \eqref{eql.ksp_gen_vl}.

\subsection{Vertex Limits}

To compute the limits of the ``$\mf{L}$"-terms, we first obtain the initial values of the integrands using Proposition \ref{thm.pnc_init}, and then we apply the integral limit property \eqref{eq.pnc_init_int} to each term.

We begin with $\mf{L}^1_\epsilon$, which we write as
\[ \mf{L}^1_\epsilon = \sum_{m = 1}^n \int_{\mc{S}_\epsilon} \sss{A}{m}^I D_3 \sss{\Phi}{m}_I = \sum_{m = 1}^n \epsilon^{-1} \int_{\mc{S}_\epsilon} \sss{B}{m}^I D_3 \sss{\Phi}{m}_I \text{.} \]
Along each null generator $\gamma_\omega$, $\omega \in \Sph^2$, we have the limits
\[ \lim_{v \searrow 0} \valat{\sss{B}{m}}_{\paren{v, \omega}} \rightarrow \sss{J}{m} \text{,} \qquad \lim_{v \searrow 0} \valat{D_3 \sss{\Phi}{m}}_{\paren{v, \omega}} \rightarrow \valat{D_3 \sss{\Phi}{m}}_p \text{,} \qquad 1 \leq m \leq n \text{.} \]
As a result, by \eqref{eq.pnc_init_int},
\[ \lim_{\epsilon \searrow 0} \mf{L}^1_\epsilon = 0 \text{.} \]
Similarly, for $\mf{L}^2_\epsilon$, we have
\[ \lim_{\epsilon \searrow 0} \mf{L}^2_\epsilon = \frac{1}{2} \sum_{m, c = 1}^n \lim_{\epsilon \searrow 0} \epsilon^{-1} \int_{\mc{S}_\epsilon} \sss{P}{cm}_{3J}{}^I \sss{B}{c}^J \sss{\Phi}{m}_I = 0 \text{.} \]

Next, for $\mf{L}^0_\epsilon$, we write
\[ \mf{L}^0_\epsilon = \frac{1}{2} \sum_{m = 1}^n \epsilon^{-2} \int_{\mc{S}_\epsilon} \brak{\epsilon \paren{\trace \ul{\chi}}} \sss{B}{m}^I \sss{\Phi}{m}_I \text{.} \]
Equation \eqref{eq.pnc_init_chib} implies $\epsilon (\trace \ul{\chi})$ converges to $-2 \vartheta_0^{-1}$ at the vertex, hence by \eqref{eq.pnc_init_int},
\[ \lim_{\epsilon \searrow 0} \mf{L}^0_\epsilon = \frac{1}{2} \vartheta_0^2 \sum_{m = 1}^n \int_{\Sph^2} \paren{-2 \vartheta_0^{-1}} \sss{J}{m}^I \valat{\sss{\Phi}{m}_I}_p d\omega = - 4 \pi \vartheta_0 \sum_{m = 1}^n \sss{J}{m}^I \valat{\sss{\Phi}{m}_I}_p \text{.} \]
Therefore, we obtain the desired limit
\[ -\lim_{\epsilon \searrow 0} \paren{\mf{L}^0_\epsilon + \mf{L}^1_\epsilon + \mf{L}^2_\epsilon} = 4 \pi \vartheta_0 \sum_{m = 1}^n \sss{J}{m}^I \valat{\sss{\Phi}{m}_I}_p \text{.} \]

As a result, equation \eqref{eql.ksp_gen_vl} becomes \eqref{eq.ksp_gen}, and the proof of \eqref{eq.ksp_gen}-\eqref{eq.ksp_gen_I} is complete.
Lastly, for the alternate representation \eqref{eq.ksp_gen_E1_ksp} of the main error terms, we simply apply Corollary \ref{thm.int_parts_pnc_lasl} to \eqref{eq.ksp_gen_E1_app}.
This concludes the proof of Theorem \ref{thm.ksp_gen}.

\begin{remark}
We can also apply the method of proof in this paper to derive the representation formula {\bf KR} of \cite{kl_rod:ksp}.
The calculation is then significantly simplified, since the term $\mf{A}^6_\epsilon$ in the proof now vanishes, as do the terms resulting from $\mf{A}^6_\epsilon$.
\end{remark}

\section{Generalization to Vector Bundles}

The representation formula {\bf KR} was further extended in \cite[Thm. 4.1]{kl_rod:ksp} to the case of covariant wave equations on sections of arbitrary vector bundles.
For convenience, we denote this extended parametrix by {\bf KRV}.
We can produce a similar extension of Theorem \ref{thm.ksp_gen} to sections of vector bundles, using a process analogous to that of \cite{kl_rod:ksp}.
We summarize this method in this section, and we connect this extended version to Theorem \ref{thm.ksp_gen} as well as to {\bf KRV}.
The final result represents a culmination of the ideas which led to {\bf KR}, {\bf KRV}, and Theorem \ref{thm.ksp_gen}.

\subsection{General Ideas}

The extension from {\bf KR} to {\bf KRV} is essential for the main application presented in \cite{kl_rod:ksp}: a gauge-invariant proof of the $L^\infty$-estimates for the Yang-Mills curvature required in Eardley and Moncrief's classical global existence result for the Yang-Mills equations in \cite{ea_mo:g_ymh}.
\footnote{In fact, the papers \cite{ea_mo:l_ymh, ea_mo:g_ymh} dealt with the more general Yang-Mills-Higgs equations.}
Recall that \cite{ea_mo:g_ymh} applied the standard representation formula for the scalar wave equation in Minkowski space and, as a result, required the Cronstr\"om gauge condition in order to control the gauge potential.
The same was also true for its curved spacetime analogue \cite{chru_sh:ym_curv}.
Using {\bf KRV}, however, one can avoid the gauge potential altogether and hence bypass the need for a favorable gauge condition.

Note that the parametrix {\bf KR} itself is not directly applicable to the Yang-Mills setting.
In order to achieve this gauge-invariance property, one requires a version of the representation formula suitable for Lie algebra-valued tensor fields and the associated gauge covariant derivatives.
Not only does the extended parametrix {\bf KRV} of \cite[Thm. 4.1]{kl_rod:ksp} address this issue, but its derivation is also completely straightforward in light of the proof of {\bf KR}.

The key observations behind this extension concern the natures of the operations used in the derivation of {\bf KR}.
These can be summarized by the following statement:
\begin{itemize}
\item On the past null cone $\mc{N}^-(p)$ for which we derive the parametrix, the (mixed) covariant derivatives on $\mc{N}^-(p)$ satisfy natural Leibniz rules and are compatible with both the spacetime metric $g$ and the induced horizontal metrics $\lambda$ on $\mc{N}^-(p)$.
In particular, this legitimizes the integration by parts operations fundamental to the derivation of {\bf KR}.
\end{itemize}
In other words, if covariant derivatives on a vector bundle behave in a manner analogous to the above statement, then the proof of {\bf KR} can be directly adapted to prove {\bf KRV}.
In particular, the gauge covariant derivatives on Lie algebra-valued tensor fields in the Yang-Mills setting satisfy a version of this property, as long as the associated metrics are modified using an appropriate positive-definite scalar product on the Lie algebra under consideration.

It is unsurprising, then, that an analogous extension can be made to the representation formula of Theorem \ref{thm.ksp_gen}.
The key observations behind this extension remain the same as before, so once the framework is set, the extension becomes completely straightforward.
Next, we will briefly summarize the basic constructions for this extension of Theorem \ref{thm.ksp_gen}.

\subsection{Basic Constructions}

Let $\mc{V}$ denote a vector bundle over $M$, and let $\mc{D}$ and $\brcmp{\cdot, \cdot}$ denote a connection and a bundle metric on $\mc{V}$, respectively.
We additionally stipulate that $\mc{D}$ and $\brcmp{\cdot, \cdot}$ satisfy the following compatibility condition: for any sections $S, T \in \Gamma \mc{V}$, we have the Leibniz-type identity
\[ D_\alpha \brcmp{S, T} = \brcmp{\mc{D}_\alpha S, T} + \brcmp{S, \mc{D}_\alpha T} \text{.} \]
Recall that $\Gamma \mc{V}$ denotes the space of all smooth sections of $\mc{V}$.

Next, we let $\ol{\mc{V}}$ denote the restriction of $\mc{V}$ to $\mc{N}^-(p)$, and we note that $\mc{D}$ induces a connection $\ol{\mc{D}}$ on $\ol{\mc{V}}$.
We will also require the following constructions:
\begin{itemize}
\item We can define various ``mixed bundles" over $\mc{N}^-(p)$, expressed as
\[ \ul{T}^k \ol{T}^r \ol{\mc{V}} = \ul{T}^k \mc{N}^-\paren{p} \otimes \ol{T}^r \mc{N}^-\paren{p} \otimes \ol{\mc{V}} \text{,} \qquad k, r \geq 0 \text{,} \]
where the right-hand side denotes tensor products of vector bundles.

\item Moreover, we can naturally construct ``mixed covariant derivatives" $\ol{\nasla}$ on the above mixed bundles, which behave like $\nasla$ on the ``horizontal" components, $\ol{D}$ on the ``extrinsic" components, and $\ol{\mc{D}}$ on $\ol{\mc{V}}$.
In particular, for a decomposable element $A \otimes B \otimes C \in \Gamma \ul{T}^k \ol{T}^r \ol{\mc{V}}$ and a vector field $X$ on $\mc{N}^-(p)$, we define the mixed covariant derivative by
\[ \ol{\nasla}_X \paren{A \otimes B \otimes C} = \nasla_X A \otimes B \otimes C + A \otimes \ol{D}_X B \otimes C + A \otimes B \otimes \ol{\mc{D}}_X C \text{.} \]

\item The metrics $\brcmp{\cdot, \cdot}$, $g$, and $\gamma$ induce pairings for the mixed bundles, which we also denote by $\brcmp{\cdot, \cdot}$, and which, similar to the original bundle metric, satisfy compatibility conditions with respect to the above mixed connections.

\item We can define the $\mc{V}$-curvature $\mc{R}$ to describe commutations of two $\mc{D}$-derivatives.
More specifically, we can define for any $T \in \Gamma \mc{V}$ and vector fields $X, Y$ on $M$ the quantity $\mc{R}_{XY}[T] \in \Gamma \mc{V}$ by the formula
\[ \mc{R}_{XY} \brak{T} = \mc{D}_X \paren{\mc{D}_Y T} - \mc{D}_Y \paren{\mc{D}_X T} - \mc{D}_{\brak{X, Y}} T \text{.} \]
\end{itemize}

We can recover the setting of {\bf KR} by taking $\mc{V}$ to be the bundle $T^r M$ of all tensors on $M$ of total rank $r$, with $\mc{D} = D$ the connection on $T^r M$ induced by the Levi-Civita connection and $\brcmp{\cdot, \cdot} = g$ the full metric contraction operation on $T^r M$.
Note that $\mc{D}$ and $\brcmp{\cdot, \cdot}$, as given here, are indeed compatible.
For this special case, the above mixed bundles and their associated connections and metrics coincide with the corresponding ``mixed" objects defined in previous sections.
This construction of mixed bundles and derivatives along with the compatibility of the associated connections and metrics are the driving forces behind the extension of Theorem \ref{thm.ksp_gen}.

We will denote the object satisfying the covariant wave equation on $\mc{V}$ to be $\Phi \in \Gamma \mc{V}$, and we denote the associated nonlinearity by $\Psi \in \Gamma \mc{V}$.
It remains now to describe the first-order coefficients of the wave equation, which we denote by $P$.
In this setting, we can define $P$ to be a ``$(\mc{V} \otimes \mc{V})$-valued $1$-form", that is, a section of the vector bundle $T^\ast M \otimes \mc{V} \otimes \mc{V}$, where $T^\ast M$ is the cotangent bundle of $M$.
In particular, the restriction of any component $P_\alpha$ to $\mc{N}^-(p)$ is a section of $\ol{\mc{V}} \otimes \ol{\mc{V}}$.

Using the bundle metric $\brcmp{\cdot, \cdot}$, we can define the natural bilinear pairings
\[ \left| \cdot, \cdot \right\rangle : \Gamma \paren{\ol{\mc{V}} \otimes \ol{\mc{V}}} \times \Gamma \ol{\mc{V}} \rightarrow \Gamma \ol{\mc{V}} \text{,} \qquad \left\langle \cdot, \cdot \right| : \Gamma \ol{\mc{V}} \times \Gamma \paren{\ol{\mc{V}} \otimes \ol{\mc{V}}} \rightarrow \Gamma \ol{\mc{V}} \text{,} \]
given for decomposable elements by
\[ \left| A \otimes B, C \right\rangle = A \cdot \brcmp{B, C} \text{,} \qquad \left\langle A, B \otimes C \right| = \brcmp{A, B} \cdot C \text{,} \qquad A, B, C \in \Gamma \ol{\mc{V}} \text{,} \]
and extended accordingly to general elements.
We will also need the maps
\begin{align*}
\brcmp{\cdot \mid \cdot \mid \cdot} &: \Gamma \ol{\mc{V}} \times \Gamma \paren{\ol{\mc{V}} \otimes \ol{\mc{V}}} \times \Gamma \ol{\mc{V}} \rightarrow C^\infty\paren{\mc{N}^-\paren{p}} \text{,} \\
\brac{\cdot, \cdot} &: \Gamma \paren{\ol{\mc{V}} \otimes \ol{\mc{V}}} \times \Gamma \paren{\ol{\mc{V}} \otimes \ol{\mc{V}}} \rightarrow \Gamma \paren{\ol{\mc{V}} \otimes \ol{\mc{V}}} \text{,}
\end{align*}
defined for decomposable elements by
\[ \brcmp{A \mid B \otimes C \mid D} = \brcmp{A, B} \brcmp{C, D} \text{,} \qquad \brac{A \otimes B, C \otimes D} = \brcmp{B, C} \cdot A \otimes D \text{,} \]
and extended as multilinear functions to general elements.

For example, consider the standard case $\mc{V} = T^r M$ of {\bf KR}.
Let $S, T \in \Gamma \ol{\mc{V}}$, let $A, B \in \Gamma (\ol{\mc{V}} \otimes \ol{\mc{V}})$, and let $I, J, K$ denote collections of $r$ extrinsic indices.
Then, the above operations are given explicitly via tensorial contractions, i.e.,
\begin{align*}
\left| A, T \right\rangle_I = A_I{}^J T_J \text{,} &\qquad \left\langle T, A \right|_J = T_I A^I{}_J \text{,} \\
\brcmp{S \mid A \mid T} = S_I A^{IJ} T_J \text{,} &\qquad \brac{A, B}_{IJ} = A_{IK} B^K{}_J \text{.}
\end{align*}

We have a natural connection $\ol{\mc{D}}$ on $\ol{\mc{V}} \otimes \ol{\mc{V}}$ induced from the connection $\ol{\mc{D}}$ on $\ol{\mc{V}}$, given by the Leibniz-type formula
\[ \ol{\mc{D}} \paren{A \otimes B} = \ol{\mc{D}} A \otimes B + A \otimes \ol{\mc{D}} B \]
for the decomposable elements and extended linearly to general sections.
Then, we can once again define mixed bundles and mixed covariant derivatives $\ol{\nasla}$ corresponding to the bundles $\ol{\mc{V}} \otimes \ol{\mc{V}}$.
It is a simple exercise to determine that the expected compatibility conditions hold for $| \cdot, \cdot \rangle$, $\langle \cdot, \cdot |$, $\langle \cdot \mid \cdot \mid \cdot \rangle$, and $\brac{\cdot, \cdot}$ with respect to the various connections and covariant derivatives.
We omit the rather tedious details.

\subsection{The Generalized Equations}

The covariant wave equation under consideration is now given by
\begin{equation}\label{eq.tensor_wave_vb} \Box^{\mc{D}}_g \Phi + \left| P, \mc{D} \Phi \right\rangle = g^{\mu\nu} \mc{D}_{\mu\nu} \Phi + \left| P_\mu, \mc{D}^\mu \Phi \right\rangle = \Psi \text{,} \end{equation}
where $\Phi, \Psi \in \Gamma \mc{V}$ and $P \in \Gamma (\mc{V} \otimes \mc{V})$.
The precise definition of $\Box^{\mc{D}}_g$ can be a bit subtle.
Here, the inner covariant derivative $\mc{D}_\beta$ is that of the connection on $\Gamma \mc{V}$, but the outer covariant derivative $\mc{D}_\alpha$ is the induced connection on $\Gamma (T^\ast M \otimes \Gamma \mc{V})$, which acts like $D$ on $T^\ast M$ and like $\mc{D}$ on $\mc{V}$.
In other words, for vector fields $X, Y$ on $M$, we define the ``second derivative" $\mc{D}_{XY} \Phi$ by
\[ \mc{D}_{XY} \Phi = \mc{D}_X \paren{\mc{D}_Y \Phi} - \mc{D}_{D_X Y} \Phi \text{.} \]

Let $B \in \Gamma \ol{\mc{V}}$ satisfy the system of transport equations
\begin{equation}\label{eq.transport_pre_vb} \ol{\nasla}_f B = -\frac{1}{2} \brak{\vartheta \paren{\trace \chi} - \frac{2}{f}} B + \frac{\vartheta}{2} \left\langle B, P_4 \right| \text{,} \qquad \valat{B}_p = J \text{,} \end{equation}
where $J$ is element of the fiber at $p$ of $\mc{V}$.
We also define $A = f^{-1} B \in \Gamma \ol{\mc{V}}$, as well as the coefficients $\nu \in \Gamma (\ol{\mc{V}} \otimes \ol{\mc{V}})$, given by
\[ \nu = - \ol{\nasla}^a P_a + \frac{1}{2} \ol{\nasla}_4 P_3 + \zeta^a P_a + \frac{1}{4} \paren{\trace \ul{\chi}} P_4 + \frac{1}{4} \paren{\trace \chi} P_3 + \frac{1}{2} \brac{P_4, P_3} \text{,} \]
We have indexed with respect to arbitrary $f$-adapted null frames on $\mc{N}^-(p)$.

We can now finally state our extension of Theorem \ref{thm.ksp_gen}:

\begin{theorem}\label{thm.ksp_gen_ex}
Suppose $\mc{V}$ is a vector bundle over $M$, with compatible connection and bundle metric $\mc{D}$ and $\brcmp{\cdot, \cdot}$.
We assume all the objects derived from $\mc{V}$, $\mc{D}$, and $\brcmp{\cdot, \cdot}$ as defined in this section, and we define $\Phi$, $\Psi$, $P$, $B$, $A$, $\nu$ as in the above development.
If $p \in M$, and if the past regular null cone $\mc{N}^-(p)$ satisfy the same assumptions as in Theorem \ref{thm.ksp_gen}, then
\begin{equation}\label{eq.ksp_gen_vector} 4 \pi \vartheta_0 \brcmp{J, \valat{\Phi}_p} = \mf{F}\paren{p; v_0} + \mf{E}^1\paren{p; v_0} + \mf{E}^2\paren{p; v_0} + \mf{I}\paren{p; v_0} \text{,} \end{equation}
where
\begin{align*}
\mf{F}\paren{p; v_0} &= - \int_{\mc{N}^-\paren{p; v_0}} \brcmp{A, \Psi} \text{,} \\
\mf{E}^1\paren{p; v_0} &= - \int_{\mc{N}^-\paren{p; v_0}} \brak{\brcmp{\ol{\nasla}^a A, \ol{\nasla}_a \Phi} + \paren{\zeta^a - \ul{\eta}^a} \brcmp{\ol{\nasla}_a A, \Phi}} \text{,} \\
\mf{E}^2\paren{p; v_0} &= \int_{\mc{N}^-\paren{p; v_0}} \mu \brcmp{A, \Phi} + \frac{1}{2} \int_{\mc{N}^-\paren{p; v_0}} \brcmp{A, \mc{R}_{43}\brak{\Phi}} \\
&\qquad - \int_{\mc{N}^-\paren{p; v_0}} \paren{\brcmp{\ol{\nasla}^a A \mid P_a \mid \Phi} + \brcmp{A \mid \nu \mid \Phi}} \text{,} \\
\mf{I}\paren{p; v_0} &= -\frac{1}{2} \int_{\mc{S}_{v_0}} \paren{\trace \ul{\chi}} \brcmp{A, \Phi} - \int_{\mc{S}_{v_0}} \brcmp{A, \mc{D}_3 \Phi} - \frac{1}{2} \int_{\mc{S}_{v_0}} \brcmp{A \mid P_3 \mid \Phi} \text{.}
\end{align*}
Again, the error terms $\mf{E}^1(p; v_0)$ can be alternately expressed as
\[ \mf{E}^1\paren{p; v_0} = \int_{\mc{N}^-\paren{p; v_0}} \brcmp{\ol{\lasl} A + 2 \zeta^a \ol{\nasla}_a A, \Phi} \text{.} \]
\end{theorem}

We can derive Theorem \ref{thm.ksp_gen_ex} in a manner completely analogous to Theorem \ref{thm.ksp_gen}.
We simply replace the compatibility properties of $\ol{D}$ and $\ol{\nasla}$ in the tensorial case by the abstract compatibility properties described above.

\subsection{Reduction to Previous Cases}

Note that we can obtain {\bf KRV} trivially from Theorem \ref{thm.ksp_gen_ex} simply by taking $P \equiv 0$.
Next, we will show how to Theorem \ref{thm.ksp_gen} can be recovered from Theorem \ref{thm.ksp_gen_ex}.

For this, we consider the direct sum vector bundle
\[ \mc{V} = T^{\sss{r}{1}} M \oplus \ldots \oplus T^{\sss{r}{n}} M \text{,} \]
with the associated metric and connection on $\mc{V}$ given naturally by
\begin{align}
\label{eq.mc_system} \brcmp{\paren{\sss{S}{1}, \ldots, \sss{S}{n}}, \paren{\sss{T}{1}, \ldots, \sss{T}{n}}} &= \sum_{m = 1}^n g\paren{\sss{S}{m}, \sss{T}{m}} \text{,} \\
\notag \mc{D} \paren{\sss{S}{1}, \ldots, \sss{S}{n}} &= \paren{D \sss{S}{1}, \ldots, D \sss{S}{n}} \text{.}
\end{align}
We then define the fundamental objects
\[ \Phi = \paren{\sss{\Phi}{1}, \ldots, \sss{\Phi}{n}} \text{,} \qquad \Psi = \paren{\sss{\Psi}{1}, \ldots, \sss{\Phi}{n}} \text{,} \]
and we similarly define $P$ from the component objects $\sss{P}{mc}$.
With these specific definitions, then the covariant wave equation \eqref{eq.tensor_wave_vb} coincides with \eqref{eq.tensor_wave_system}, and formula of Theorem \ref{thm.ksp_gen_ex} reduces to Theorem \ref{thm.ksp_gen}.

Note that the vector bundle formalism of \eqref{eq.ksp_gen_vector} automatically encapsulates the fact that Theorem \ref{thm.ksp_gen} deals with a system of tensor wave equations.
As a result, the parametrix {\bf KRV} can also deal with systems of covariant wave equations in this fashion.
What is entirely new with respect to \cite{kl_rod:ksp} is the presence of first-order terms, the resulting coupling, and the weakening of the assumptions.

\begin{remark}
Alternatively, one could possibly reinterpret the first-order coefficients $P$ by modifying the given connection $\mc{D}$ to also take into account contributions from $P$.
This would bring us closer to the case of {\bf KRV}.
However, such modified connections would generally fail to be compatible with the given metric, hence one would need to deal with error terms resulting from this.
\end{remark}

\subsection{Applications to General Relativity: Einstein-Yang-Mills Spacetimes}

We conclude this paper by briefly summarizing a couple of applications for which Theorem \ref{thm.ksp_gen_ex} could be helpful.
First, we can apply Theorem \ref{thm.ksp_gen_ex} to treat the analogue of Theorem \ref{thm.bdc_maxwell} for the Einstein-Yang-Mills setting.
This is quite analogous to the Einstein-Maxwell problem discussed in \cite{shao:bdc_nv}, except that the matter field $\mf{F}$ is now a $\mf{g}$-valued $2$-form, where $\mf{g}$ is the Lie algebra associated with the Yang-Mills theory.
In other words, we can think of $\mf{F}$ as a section of the bundle $T^2 M \otimes \mf{g}$.

Recall also that we can define ``gauge covariant" derivatives on such $\mf{g}$-valued tensor bundles, which we denote by $\mc{D}$.
Moreover, if we assume as is standard that $\mf{g}$ admits an invariant positive-definite scalar product $\brcmp{\cdot, \cdot}$, then we can combine $\brcmp{\cdot, \cdot}$ and $g$ to define bundle metrics on such $\mf{g}$-valued tensor bundles which are compatible with $\mc{D}$.
Hence, we retrieve a setting that can be described using the language of Theorem \ref{thm.ksp_gen_ex}.

Similar to the Einstein-Maxwell setting, we again have a system of two covariant wave equations in the Einstein-Yang-Mills analogue, one for the spacetime curvature $R$, and one for the derivative of the Yang-Mills curvature, $\mc{D} \mf{F} \in \Gamma (T^3 M \otimes \mf{g})$.
Algebraically, these equations are analogous in form to the system \eqref{eq.wave_curv_em} treated in the Einstein-Maxwell setting.
Unlike in \cite{shao:bdc_nv}, we cannot apply Theorem \ref{thm.ksp_gen} here, since the wave equation satisfied by $\mc{D} \mf{F}$ involve gauge covariant derivatives rather than the standard tensorial derivatives.
\footnote{In terms of the usual tensorial covariant derivative, such a wave equation for $\mf{F}$ would of course fail to be gauge-independent.}
We can, however, describe this setting in terms of Theorem \ref{thm.ksp_gen_ex}.
In fact, we simply take
\[ \mc{V} = T^4 M \oplus \paren{T^2 M \otimes \mf{g}} \text{,} \]
and we construct a natural connection and bundle metric on $\mc{V}$ from the above developments using the same method as in \eqref{eq.mc_system}.
Consequently, the system of wave equations described above is now of the form \eqref{eq.tensor_wave_vb}, and we can apply Theorem \ref{thm.ksp_gen_ex} in order to derive the necessary uniform bounds for $R$ and $\mc{D} \mf{F}$.

\subsection{Applications to General Relativity: Time Foliations}

In \cite{wang:ibdc, wang:tbdc}, and to a lesser extent in \cite{shao:bdc_nv}, one requires the observation that the second fundamental form $k$ of the given time foliation satisfies a covariant wave equation.
This can be cleanly described in terms of the language of this section.
Furthermore, with this characterization, one can more easily derive the energy estimates for $k$ associated with this wave equation, which are crucial in the above works.

In this setting, the spacetime $M$ is given as a smooth $1$-parameter family of spacelike hypersurfaces $\{\Sigma_\tau\}$, and $k$ denotes the second fundamental forms on the $\Sigma_\tau$'s.
Then, the bundle $\mc{V}$ under consideration here is the ``horizontal bundle" of all symmetric $2$-tensors in $M$ which are tangent to the $\Sigma_\tau$'s; note that $k \in \Gamma \mc{V}$.
For the bundle metric, we can simply consider the pullback $\gamma$ of $g$ to the $\Sigma_\tau$'s, i.e., the induced metrics of the $\Sigma_\tau$'s.
The compatible connection on $\mc{V}$ can be defined in the same manner as for the connections on the horizontal bundles of past null cones.

From the Einstein equations, one can see that $k$ satisfies a covariant wave equation of the form \eqref{eq.tensor_wave_vb}.
As a result, we can apply Theorem \ref{thm.ksp_gen_ex} to $k$ in this setting.
In fact, the application of a representation formula to $k$ is an essential component in the arguments of \cite{wang:ibdc, wang:tbdc}.
Note that in contrast to the previous example and to Theorem \ref{thm.ksp_gen}, the bundle metric in this case is in fact positive-definite.

Recall that in the case of scalar wave equations, one can derive energy estimates using energy-momentum tensors in a completely standard fashion.
Moreover, this methodology extends directly to the setting of an arbitrary vector bundle along with a connection and a compatible positive-definite bundle metric.
Therefore, by expressing $k$ in the form \eqref{eq.tensor_wave_vb}, then one can immediately derive associated global and local energy estimates for $k$.

\bibliographystyle{amsplain}
\bibliography{bib}

\end{document}